\documentclass[11pt]{article}

\usepackage[utf8]{inputenc}
\usepackage{amsmath, amssymb, amsfonts, mathrsfs, bbm, bbold, array, booktabs, multirow, amsthm, ulem, indentfirst}
\usepackage{graphicx}
\usepackage{bigints}
\usepackage{caption}
\usepackage{geometry}
\usepackage{hyperref}
\usepackage{times}
\usepackage{float}
\usepackage{xcolor}
\usepackage{ragged2e}
\usepackage{csquotes}
\usepackage{longtable}
\usepackage{caption}
\usepackage{ragged2e}
\usepackage{url}
\theoremstyle{plain}
\newtheorem{theorem}{Theorem}[section]
\newtheorem{proposition}[theorem]{Proposition}
\newtheorem{lemma}[theorem]{Lemma}
\newtheorem{corollary}[theorem]{Corollary}

\theoremstyle{definition}
\newtheorem{definition}[theorem]{Definition}

\theoremstyle{remark}
\newtheorem{remark}[theorem]{Remark}

\title{Hypergeometric Series Representations\\ for the Perimeter of Lam\'e Superellipses}
\author{R. Omar Rodriguez\textsuperscript{1} and Yomber Montilla\textsuperscript{2}\\
  \small \textsuperscript{1} Decanato de Ciencias y Tecnolog\'ia, Universidad Centroccidental Lisandro Alvarado, Barquisimeto 3001, Venezuela. \\
\small \textsuperscript{2} Universidad Técnica Estatal de Quevedo, Facultad de Ciencias de la Ingeniería, Quevedo, Ecuador. \\
\small email: romar4473@gmail.com; ymontillal@uteq.edu.ec.
}

\begin{document}
\date{}
\maketitle

\begin{abstract}
We derive exact analytic representations for the perimeter of a Lamé superellipse of degree $s>0$. The result is expressed in terms of two branches defined by series whose terms are Gauss hypergeometric functions: a negative branch for $0<s<1$ and a positive branch for $s>1$. For the positive branch, the convergence condition follows from the Leibniz test; the negative branch, although divergent in the ordinary sense, is shown to be Abel-summable. Consistently with the symmetry under interchange of the semi-axes, the formula is invariant under axis permutation. As $s$ varies, the family interpolates between the Lamé cross and the rectangle, while the case $s=1$ corresponds to the rhombus, which acts as the transition curve with the shortest perimeter within the family.

\end{abstract}

\section{Introduction.}

The superellipse, also known as a Lam\'e curve, is defined in polar
coordinates as follows
\begin{equation}\label{1-r_s(a,b)}
r_{a,b}(\theta)=
\left(
\frac{|\cos\theta|^s}{a^s}
+
\frac{|\sin\theta|^s}{b^s}
\right)^{-1/s}\ ,\qquad 0\leq\theta<2\pi\ ,
\end{equation}
where $a$, $b$, and $s$ are positive parameters that play different roles in the geometric configuration of the curves. Specifically, the parameters $a$ and $b$ represent the major and minor semi-axes and determine the size of the curve, whereas the exponent parameter $s$ determines its characteristic shape. Smaller values of $s$ lead to peak- or corner-shaped structures, while larger values produce smoother, more rounded geometries \cite{Lame1818, Gardner1977}. Thus, this formula defines a continuous spectrum of closed curves ranging from star-shaped geometries to rectangles, with elliptical forms appearing as intermediate cases.

This geometric flexibility makes the superellipse especially well suited for a wide variety of applications. In granular mechanics, statistical physics, and computational fluid dynamics \cite{Grohn2023, Wang2024, Yuan2020, KornickFranklin2021, Colt2022, Mizani2025}, superellipses have been used to model convex nonspherical particles. In optics and electromagnetism, they have been employed in the design of diffractive apertures, waveguides, and microstrip transitions \cite{SanchezBrea2024, IEEE2024, MDPI2020}. In solid-state physics, the superellipse has been proposed as a model for the spontaneous magnetization of ferromagnets \cite{Perevertov2025}, in medical physics for the characterization of radiotherapy fields \cite{MendezCasar2021, Mendez2026}, and in photovoltaic engineering for modeling solar-panel I-V curves \cite{OlayiwolaChoi2023}.

From a mathematical standpoint, the most significant advances can be organized along two lines. The first concerns the geometric properties of the curve: Ref.\cite{Gridgeman1970} systematized the basic properties of Lam\'e curves, and, more recently, \cite{FarhadianPonomarenko2023} studied the constant $\pi_p$ — the perimeter of the supercircle of exponent $p$ divided by its diameter — establishing extremal properties and its analogy with the golden ratio. The second line concerns the connection with special functions: Ref.\cite{FiedorowiczRamalingam2024, FiedorowiczRamalingam2026} proved that the area enclosed by the Lamé curve $r=(\cos^{2n}\theta+\sin^{2n}\theta)^{-1/2n}$ is proportional to the arc length of the associated sinusoidal spiral $r^n=\cos(n\theta)$. Both quantities reduce to the same Euler integral, which can be expressed as a value of $_2F_1$ on the unit circle, establishing the most direct bridge between the geometry of Lamé curves and hypergeometric functions. In particular, for $n=2$
this reduces to the well-known relation between the squircle $r=(\cos^{4}\theta+\sin^{4}\theta)^{-1/4}$ and the lemniscate of Bernoulli $r^2=\cos(2\theta)$.

These advances, however, have not led to a closed-form expression for the general perimeter of the superellipse. The case $s=2$ is the most thoroughly developed: Ref.\cite{Abbott2009} showed that the exact representations of Maclaurin, Gauss-Kummer, and Euler are related through quadratic transformations of $_2F_1$, and obtained a formula symmetric in $a$ and $b$ in terms of the Legendre function $P_{1/2}$. In the isotropic case $a=b$, the corresponding supercircle was studied in Ref.\cite{MontillaRodriguezLinares2026}, where an exact arc-length formula was derived and related to a hypergeometric representation of $\pi$. This result provides a natural consistency check for the present formulation when the semi-axes coincide. For $s>1$, \cite{Tanaka} presented an unpublished technical report deriving the arc length as a double series in Pochhammer symbols. The expression is mathematically consistent within its domain of validity and was verified with numerical examples for $s = 5/2$, $s = 2$, and $s = 1$. \cite{Erbas2022} proposed a closed-form approximation valid over a broader range, achieving a relative error below $\sim10^{-4}$.

Part of the difficulty in obtaining a general expression for the perimeter of a superellipse lies in identifying a partition angle that allows the curve, in the first quadrant, to be decomposed into two arcs whose lengths admit tractable integral representations. As shown in \cite{MontillaRodriguezLinares2026}, in the isotropic case $a=b$, corresponding to supercircles, this angle is $\pi/4$. In the present work, we determine the partition angle for the anisotropic case $a\neq b$ through an analysis of the behavior of the curvature. We also obtain an exact representation for the perimeter of the superellipse in terms of a series expansion whose terms involve Gauss hypergeometric functions with parameters depending on the summation index, in accordance with the type of expansions considered in \cite{IshkhanyanShahverdyanIshkhanyan2014}. The resulting representation is conditionally convergent for $s>1$, whereas for $0<s<1$ it is interpreted through Abel regularization.

The paper is organized as follows. In Section \ref{Section2}, we briefly review the geometric properties of superellipses that are needed to introduce the partition angle. In Sections \ref{Section3} and \ref{Section4}, we derive an exact representation for the perimeter of a Lam\'e curve with $s>0$, $s\neq 1$. Owing to the symmetry under the interchange of the semi-axes, it is sufficient to consider the case $a\geq b$; the case $b\geq a$ follows analogously. The resulting formula is symmetric with respect to the interchange $a\leftrightarrow b$ and decomposes into two branches separated by the critical value $s=1$. In Section \ref{Section4}, we also analyze the convergence and divergence properties of this representation. As a consistency check, several well-known special cases are recovered in Section \ref{Section5}.

In Section \ref{Section6}, we study the behavior of the perimeter as a function of the parameter $s$. We show that the family interpolates from the Lam\'e cross, in the limit $s\to 0^+$, through the rhombus, as $s\to 1^\pm$, to the rectangle, as $s\to\infty$. This analysis allows us to identify the extremal values that bound the perimeter within the family under consideration. Finally, in Section \ref{Conclusions}, we present the conclusions. The appendices contain complementary technical details concerning the consistency of the approach, the hypergeometric properties used throughout the paper, and the completeness of the proofs.

\section{Preliminaries.}\label{Section2}

\begin{definition}{(Lam\'e superellipse)} 
Let $s$, $a$, $b\in \mathbb{R}^+$ and $k\in\mathbb{Z}$. Then, the superellipses form a three-parameter family of closed planar curves whose points satisfy
\begin{equation}\label{r_s(a,b)}
\mathscr{C}_s(a,b)
=
\left\{(r,\theta)\in[0,\infty)\times[0,2\pi):\ 
r=r_{a,b}(\theta)\right\}\ ,
\end{equation}
where $r_{a,b}(\theta)$ is given by (\ref{1-r_s(a,b)}). 

We recall the following elementary geometric properties of superellipses.
\begin{remark} The shape of $\mathscr{C}_s$ depends on $s$ as follows.
\begin{itemize}
  \item[(i)] The curves $\mathscr{C}_s$ have four corners located at $\theta_k = k\pi/2$.
  \item[(ii)] If $0<s<1$, the corners are peak-like and the sides are concave toward the origin.
  \item[(iii)] If $s=1$, the vertices are ordinary corners and the sides are straight diagonal lines.
  \item[(iv)] If $s>1$, the corners are rounded and the sides are convex away from the origin.
\end{itemize}

\end{remark}
\end{definition}

\begin{remark} The parameters $a$ and $b$ determine the semi-axes and axial symmetries of $\mathscr{C}_s$ in the following sense.
\begin{itemize}
\item[(i)] If $a\neq b$, the curves $\mathscr{C}_s$ have two semi-axes at
\begin{equation}
     r(k\pi)=a \quad\text{and}\quad r\left(\frac{2k+1}{2}\pi\right)=b\ ,\quad k\in\mathbb{Z}\ .
\end{equation}
The semi-major and semi-minor axes of $\mathscr{C}_s(a,b)$ have lengths $\max\{a,b\}$ and $\min\{a,b\}$, respectively. Thus, if $a>b$, the semi-major axis is aligned with the horizontal direction, whereas if $b>a$, it is aligned with the vertical direction.

\item[(ii)] If $a\neq b$, the curve has elliptical symmetry and is invariant under the discrete rotation
\begin{equation}
\theta \mapsto \theta+k\pi,\qquad k\in\mathbb{Z}\ .
\end{equation}
In the particular case $a=b$, the curves exhibit circular symmetry and are invariant under
\begin{equation}
\theta \mapsto \theta+\frac{k\pi}{2},\qquad k\in\mathbb{Z}\ .
\end{equation}

\item[(iii)] The curve is invariant under axial reflection symmetries with respect to the coordinate axes. In polar form, these symmetries are given by

\begin{equation}
\theta\mapsto -\theta
\qquad\text{and}\qquad
\theta\mapsto \pi-\theta\ .
\end{equation}
Consequently, the four quadrant arcs have the same length.
\end{itemize}
\end{remark}

\begin{remark}{ (Symmetry of the perimeter)} The perimeter of $\mathscr{C}_s(a,b)$ provides a natural example of a geometric quantity depending on the three parameters $s$, $a$, and $b$. Although the interchange of the semi-axes $a$ and $b$ does not leave the curve fixed as a subset of the plane, it maps $\mathscr{C}_s(a,b)$ onto a congruent curve, namely $\mathscr{C}_s(b,a)$, through the interchange of the coordinate axes. Since arc length is invariant under such isometries, the perimeter satisfies
\begin{equation}
    L_{(s)}(a,b)=L_{(s)}(b,a).
\end{equation}
\end{remark}

To the best of our knowledge, no closed-form analytic expression for $L_{(s)}(a,b)$ valid throughout the full parameter space has been reported so far.

Before proceeding, two remarks are in order. First, since the basic geometric properties of superellipses are well known, the following subsections are not intended to rederive them from first principles. Instead, relying on the axial symmetries of these curves, we use them to verify the occurrence of corner-type behavior and to identify the corresponding concavity regimes. Second, throughout this article, and within the scope of the present study, we assume without loss of generality that $a\geq b$. This convention is justified by the symmetry of the problem under the exchange of the semi-axes and is discussed in detail in Section (\ref{Section4}).

\subsection{Peaks, corners and smoothness.}\label{R-R}

As noted above, by axial symmetry it is enough to analyze the possible loss of smoothness at the vertices corresponding to $\theta_k=k\pi/2$, with $k=0,\ldots,3$.
\begin{lemma}[Regularity at the vertices]\label{Lemma-2.5} 
Let $a$,$b>0$ and $s>0$. At each vertex $\theta_k=k\pi/2$, the superellipse $\mathscr{C}_s(a,b))$ satisfies the following classification.
\begin{itemize}
    \item[(i)] If $s>1$, the curve is smooth, $C^1$, at $\theta_k$.
    \item[(ii)] If $s=1$, the curve has an angular point, that is, a finite but discontinuous derivative.
\item[(iii)] If $0<s<1$, the curve has a cusp, that is, infinite one-sided derivatives with opposite
signs.
\end{itemize}
\end{lemma}

\begin{proof}
We analyze the local regularity of the curve at the vertices by comparing the corresponding one-sided derivatives of its real-valued polar representation. Since $s>0$ is arbitrary, the polar equation must be written separately in each quadrant, so that all powers are taken over nonnegative quantities. This piecewise representation allows us to identify, on each angular interval, the branch containing the vertex under consideration. Thus, for $\theta\in(0,\pi/2]$ and $\theta\in[\pi/2,\pi]$, we consider the branches 
\begin{equation}\label{r_1,r_2}
r_{\mathrm{I}}=
\left(
\frac{\cos^s\theta}{a^s}
+
\frac{\sin^s\theta}{b^s}
\right)^{-1/s}
\quad\text{and}\quad r_{\mathrm{II}}=
\left(
\frac{(-\cos\theta)^s}{a^s}
+
\frac{\sin^s\theta}{b^s}
\right)^{-1/s}\ ,
\end{equation}
respectively, whereas for $\theta\in[\pi,3\pi/2]$ and $\theta\in[3\pi/2,2\pi]$, we consider 
\begin{equation}
r_{\mathrm{III}}=
\left(
\frac{(-\cos\theta)^s}{a^s}
+
\frac{(-\sin\theta)^s}{b^s}
\right)^{-1/s}
\qquad\text{and} \qquad
r_{\mathrm{IV}}=
\left(
\frac{\cos^s\theta}{a^s}
+
\frac{(-\sin\theta)^s}{b^s}
\right)^{-1/s}
\end{equation}
respectively.

Thus, the branches $r_{\mathrm{I}}$ and $r_{\mathrm{II}}$ meet at the vertex corresponding to $\theta=\pi/2$, whereas $r_{\mathrm{I}}$ and $r_{\mathrm{IV}}$ meet at the vertex corresponding to $\theta=0$, the latter being identified with $\theta=2\pi$ by periodicity. Similarly, the vertices corresponding to $\theta=\pi$ and $3\pi/2$ are obtained from the previous cases by symmetry and periodicity. Therefore, it is enough to compare the adjacent one-sided derivatives at $\theta=0$ and $\pi/2$. We obtain
\begin{equation}\label{r^prime1}
   \lim_{\theta\rightarrow 0^+}r^\prime_{\mathrm{I}}(\theta) -\lim_{\theta\rightarrow 2\pi^-}r^\prime_{\mathrm{IV}}(\theta)=
\begin{cases}
0\ , & s>1\ , \\
-2a^2/b\ ,   & s=1\ , \\
-\infty\ ,  & 0<s<1
\end{cases} 
\end{equation}
and
\begin{equation}\label{r^prime2}
\lim_{\theta\rightarrow \pi/2^+} r^\prime_{\mathrm{II}}\left(\theta\right)-\lim_{\theta\rightarrow \pi/2^-}  r^\prime_{\mathrm{I}}\left(\theta\right)=
\begin{cases}
0\ , & s>1\ , \\
-2b^2/a\ ,   & s=1\ , \\
-\infty\ ,  & 0<s<1\ .
\end{cases}
\end{equation}

Equations (\ref{r^prime1}) and (\ref{r^prime2}) give precisely the three cases stated in the lemma. The remaining vertices are obtained from these two cases by axial symmetry and periodicity. Hence, the classification holds for every $\theta_k=k\pi/2$, with $k=0,\ldots,3$, and the proof is complete.
\end{proof}

\subsection{Concavity and the transition angle.}

Because of the symmetries of the family of $\mathscr{C}_s$-curves, it is enough to analyze the signed curvature in the first quadrant, where $r=r_{\mathrm I}$ is given by (\ref{r_1,r_2}). For a polar curve, the signed curvature with respect to the origin is given by
\begin{equation}
k_s=\frac{r^2+2 r^{\prime 2}-r r^{\prime\prime}}{(r^2+r^{\prime 2})^{3/2}}\ .
\end{equation}
A direct computation gives
\begin{equation}
    k_s(\theta)=(s-1)\frac{(ab)^{2s}}{2^{s-2}}\frac{(\sin2\theta)^{s-2}}{r^{s+1}}\left(b^{2s}(\cos\theta)^{2(s-1)}+a^{2s}(\sin\theta)^{2(s-1)}\right)^{-3/2}\ .
\end{equation}
Since $\theta\in(0,\pi/2)$, all factors in this expression are strictly positive except $(s-1)$. Therefore, the sign of the signed curvature is completely determined by $(s-1)$. This yields the following lemma.
\begin{lemma}[Concavity regimes]\label{Lemma-2.6}
     Let $a,b>0$ and $s>0$. In the first quadrant $\text{sgn}(k_s(\theta))=\text{sgn}(s-1)$. In particular: $(i)$ if $s>1$, all four arcs are convex; $(ii)$ if $s=1$, $ks=0$ (straight sides);
$(iii)$ if $0<s<1$, all four arcs are concave toward the origin. 
\end{lemma}
The result follows directly from the preceding computation in the first quadrant and extends to the remaining sides by axial symmetry and periodicity.
\begin{figure}[H]
\centering
\includegraphics[width=0.85\linewidth]{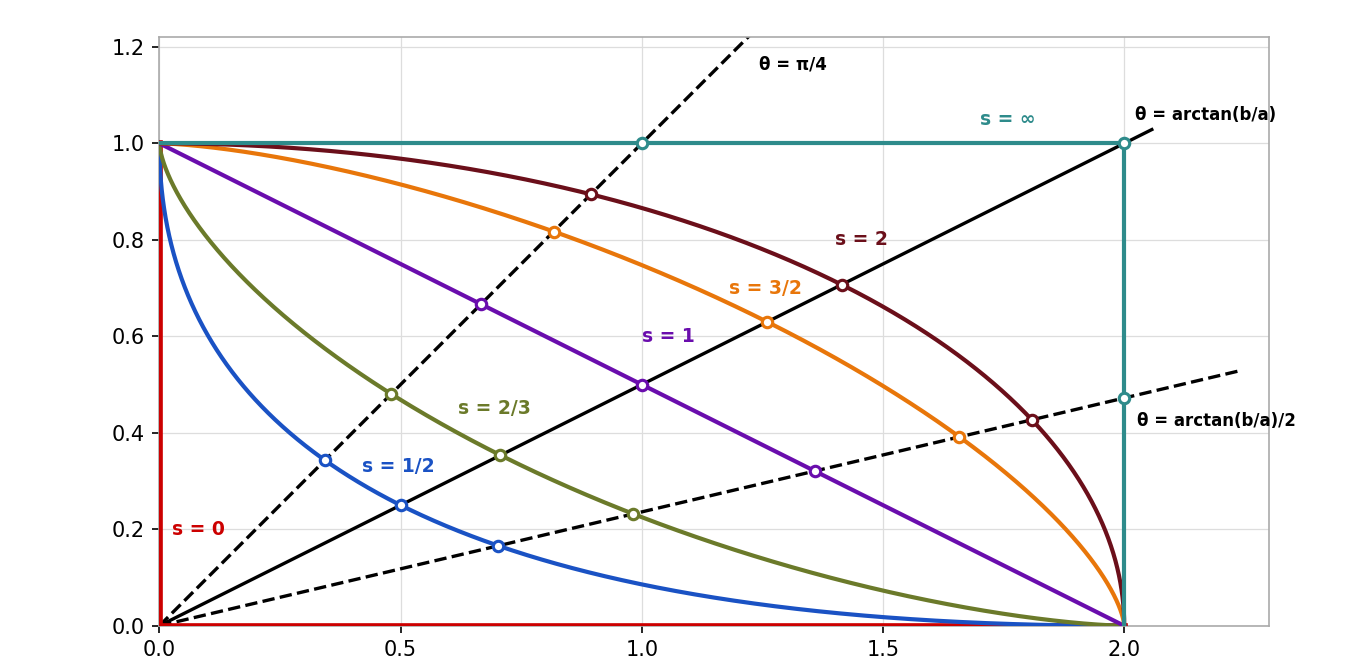}
\caption{Plots of 
several curves for $0<s<1$ and $s>1$ and three radial lines associated with $\theta=\theta_0/2,\ \theta_0$ and $\pi/4$.}
\label{Primera imagen}
\end{figure}

The previous lemma characterizes the concavity of the Lamé curve in terms of the parameter $s$ and the angular position $\theta$. However, the plots in Fig.~\ref{Primera imagen} suggest that this behavior is not distributed uniformly throughout the first quadrant. Instead, there exists a distinguished direction that separates two angular sectors in which the limiting arcs exhibit different asymptotic tendencies. This direction is determined only by the semi-axes $a$ and $b$, and therefore provides a natural geometric reference for comparing the concavity above and below it. We formalize this observation in the following remark.
\begin{remark}[The transition angle]\label{Remark2.7}
Let
\begin{equation}
\theta_0=\arctan(b/a).
\end{equation}
Then $\theta_0$ has the following properties.
\begin{enumerate}
    \item[(i)] \textit{Diagonal direction.}
    Since $\tan\theta_0=b/a$, the ray $\theta=\theta_0$ is the diagonal direction associated with the rectangle determined by the semi-axes $a$ and $b$. Equivalently, every point on this ray satisfies $y/x=b/a$.
    \item[(ii)] \textit{Radial value.}
    Along this direction, the polar radius is given by
    \begin{equation}
r(\theta_0)=2^{-1/s}\sqrt{a^2+b^2}.
    \end{equation}
    Consequently,
    \begin{equation}
    \lim_{s\to 0^+}r(\theta_0)=0,\qquad\lim_{s\to\infty}r(\theta_0)=\sqrt{a^2+b^2}.
    \end{equation}
    \item[(iii)] \textit{Sector decomposition.}
    The angle $\theta_0$ divides the first quadrant into two regions: the upper sector and the lower sector, respectively given by $\pi/2>\theta>\theta_0$ and $\theta_0>\theta>0$.
    Hence, $\theta_0$ provides a reference direction for comparing the behavior of the arcs located above and below the diagonal ray.
    \item[(iv)] \textit{Curvature along the transition direction.}
    Along $\theta=\theta_0$, the curvature reduces to
    \begin{equation}
    k_s(\theta_0)
    =
    (s-1)2^{1+1/s}
    \frac{ab}{(a^2+b^2)^{3/2}}.
    \end{equation}
    Therefore, $k_s(\theta_0)<0$ for $0<s<1$, $k_s(\theta_0)=0$ for $s=1$, and $k_s(\theta_0)>0$ for $s>1$. Moreover,
    \begin{equation}
    k_s(\theta_0)\to -\infty
    \quad \text{as} \quad s\to 0^+,
    \qquad
    k_s(\theta_0)\to +\infty
    \quad \text{as} \quad s\to\infty.
    \end{equation}
\end{enumerate}
\end{remark}

Thus, $\theta_0$ admits a clear geometric interpretation: it is the direction along which the transition between the two limiting geometries is organized. As $s\to0^+$, the superellipse collapses toward an $L$-shaped limit: the arcs above $\theta_0$ approach the vertical segment, whereas those below $\theta_0$ approach the horizontal segment. In contrast, as $s\to\infty$, the curve approaches the bounding rectangle: the arcs above $\theta_0$ become asymptotically horizontal, while those below $\theta_0$ become asymptotically vertical. Along the transition ray itself, the limiting value $r(\theta_0)\to\sqrt{a^2+b^2}$ corresponds to the vertex $(a,b)$, as shown in Fig.~\ref{Primera imagen}. This behavior is consistent with the divergence of $k_s(\theta_0)$ and is further illustrated by the representative rays displayed in Fig.~\ref{Segunda imagen}. The representative rays $\theta=\pi/4$ and $\theta=\theta_0/2$, shown in Fig.~\ref{Segunda imagen}, illustrate this sectorial behavior: they serve only as auxiliary directions for comparing the arcs above and below the transition angle $\theta_0$.

\begin{figure}[H]
\centering
\includegraphics[width=0.85\linewidth]{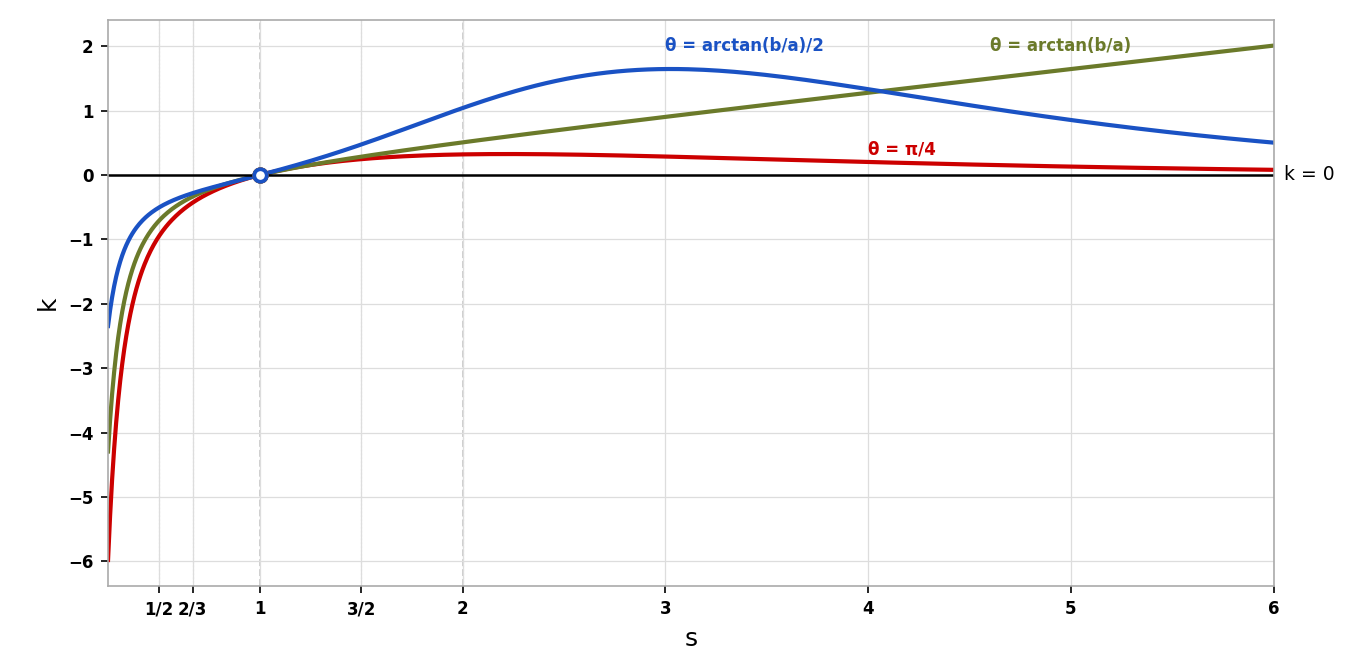}
\caption{Plots of curvature $k_s(\theta)$ along the radial lines $\theta=\theta_0/2,\ \theta_0$ and $\pi/4$.}
\label{Segunda imagen}
\end{figure}

Hence, $\theta_0$ is not merely a reference angle; it defines a natural partition of the curve into two arcs with distinct geometric behavior. This partition will also be relevant for the perimeter problem, where the arc length is naturally decomposed into contributions associated with the regions above and below the transition direction.

\section{Perimeter of the Superellipse.}\label{Section3}

The superellipse $\mathscr{C}_s(a,b)$ is invariant under axial reflections with respect to the coordinate axes. Hence, its total perimeter is four times the length of the arc contained in the first quadrant. Therefore, for a polar representation, $r=r(\theta)$, the total length is given by
\begin{equation}\label{Lquadrant}
L=4\int_0^{\pi/2}\sqrt{r^2+r^{\prime 2}}\ d\theta.
\end{equation}
\begin{theorem}\label{Thm-3.1}
    Let $s$, $a$, and $b > 0$. The perimeter of the superellipse $\mathscr{C}_s(a,b)$ is given by
    \begin{equation}\label{L-general}
L_{(s)}=
4\int_0^{\pi/2}
\left(
\frac{\cos^s\theta}{a^s}
+
\frac{\sin^s\theta}{b^s}
\right)^{-(s+1)/s}
\left(
\frac{(\cos\theta)^{2(s-1)}}{a^{2s}}
+
\frac{(\sin\theta)^{2(s-1)}}{b^{2s}}
\right)^{1/2}
d\theta.
\end{equation}
\end{theorem}
\begin{proof}
The result follows by direct substitution of the polar representation  (\ref{1-r_s(a,b)}) into the arc-length formula (\ref{Lquadrant}).
\end{proof}

\begin{theorem}[Hypergeometric series representation]\label{Thm-3.2}
Let $a\geq b>0$, $s>0$, and $s\neq 1$. For notational convenience, set
\begin{equation}
\alpha:=1+\frac{1}{s},
\end{equation}
and, for $k\in\mathbb{N}_0$, define
\begin{equation}\label{ParameterS>1}
v_k^{+}=\frac{1+2k(s-1)}{s},
\qquad
u_k^{+}=\frac{sk+1}{2(s-1)},
\end{equation}
\begin{equation}\label{Eq-parameters-minus}
v_k^{-}=1+\frac{2k(1-s)}{s},
\qquad
u_k^{-}=\frac{s(k+1)}{2(1-s)}.
\end{equation}
Then the perimeter of $\mathscr{C}_s(a,b)$ admits the following branch representations
\begin{itemize}
    \item[(A)] Positive branch. If $s>1$, then $L_{(s)}=L^+_{(s)}$ where
\begin{equation}\label{L-final-ngreater1}
\begin{split}
L^+_{(s)}&=\frac{4a}{s}
\sum_{k=0}^\infty
\frac{(-1)^k (-1/2)_k}{k!}(b/a)^{2k}
\frac{1}{v^+_k}
\,{}_2F_1\left(\alpha,v^+_k;v^+_k+1;-1\right)\\
&+
\frac{2b}{s-1}
\sum_{k=0}^\infty
\frac{(-1)^k(\alpha)_k}{k!}
\frac1{u^+_k}
\,{}_2F_1\left(-1/2,u^+_k;u^+_k+1;-a^2/b^2\right)\ .
\end{split}
\end{equation}

    \item[(B)] Negative branch. If $0<s<1$, then $L_{(s)}=L^-_{(s)}$ where 
\begin{equation}\label{L-final-nless1}
\begin{aligned}
L^-_{(s)}&=
\frac{4a}{s}
\sum_{k=0}^\infty
\frac{(-1)^k(-1/2)_k}{k!} (b/a)^{2k}
\frac{1}{v^-_k}
\,{}_2F_1\left(\alpha,v^-_k;v^-_k+1;-1\right)\\
& +\frac{2b}{1-s}
\sum_{k=0}^\infty
\frac{(-1)^k(\alpha)_k}{k!}
\frac1{u^-_k}
\,{}_2F_1\left(-1/2,u^-_k;u^-_k+1;-a^2/b^2\right)\ ,
\end{aligned}
\end{equation}
with the equality understood in the Abel-regularized sense.
\end{itemize}
\end{theorem}
The case $s=1$ is excluded from the statement and is recovered separately by taking the appropriate limiting value.

\begin{proof}
As established in Remark~2.7, the transition angle $\theta_0$ provides a natural partition of the first-quadrant arc. Hence, using Theorem~3.1, we split the perimeter integral over the intervals $[0,\theta_0]$ and $[\theta_0,\pi/2]$. In the first interval, we introduce
\begin{equation}
t_1=(b/a)^{-1}\tan\theta,
\end{equation}
whereas in the second interval we use
\begin{equation}
t_2=(b/a)\cot\theta.
\end{equation}
Both substitutions map the corresponding intervals onto $[0,1]$; in the second case, the orientation is reversed and the limits are interchanged accordingly. Thus, the perimeter can be written in the normalized form
\begin{equation}
L_{(s)}
=
4b\int_0^1
\frac{\sqrt{1+(b/a)^{-2}t_1^{2(s-1)}}}
{(1+t_1^s)^{(s+1)/s}}
\,dt_1
+
4a\int_0^1
\frac{\sqrt{1+(b/a)^2t_2^{2(s-1)}}}
{(1+t_2^s)^{(s+1)/s}}
\,dt_2 .
\end{equation}

We now distinguish the two analytic regimes $s>1$ and $0<s<1$, since the admissible binomial expansions depend on the sign of $s-1$.

\medskip

\noindent\textbf{A. Positive branch: $\mathbf{s>1}$.}

In this regime, the powers $t_1^{2(s-1)}$ and $t_2^{2(s-1)}$ are regular at the origin. We expand only the admissible factors using the generalized binomial theorem
\begin{equation}
(1+z)^\lambda
=
\sum_{k=0}^{\infty}
\frac{(-1)^k(-\lambda)_k}{k!}z^k,
\qquad |z|<1.
\end{equation}
Term-by-term integration is first justified in the corresponding domains of convergence of the binomial expansions; outside these domains, the resulting hypergeometric factors are understood by analytic continuation, as discussed in Appendix~\ref{AppendixA}.

After applying the admissible binomial expansions and using the changes of variables
\begin{equation}
z_1=t_1^{2(s-1)},
\qquad
z_2=t_2^{s},
\end{equation}
we obtain
\begin{equation}
\begin{aligned}
L^+_{(s)}
={}&
\frac{2b}{s-1}
\sum_{k=0}^{\infty}
\frac{(-1)^k(\alpha)_k}{k!}
\int_0^1
z_1^{u_k^+-1}
\left(1+\frac{a^2}{b^2}z_1\right)^{1/2}
\,dz_1
\\
&+
\frac{4a}{s}
\sum_{k=0}^{\infty}
\frac{(-1)^k(-1/2)_k}{k!}
\left(\frac{b}{a}\right)^{2k}
\int_0^1
z_2^{v_k^+-1}
(1+z_2)^{-\alpha}
\,dz_2 .
\end{aligned}
\end{equation}
Using the Euler-type integral representations
\begin{equation}
\int_0^1 z^{p-1}(1+\beta z)^{1/2}\,dz
=
\frac{1}{p}\,
{}_2F_1\left(-\frac12,p;p+1;-\beta\right),
\end{equation}
and
\begin{equation}
\int_0^1 z^{p-1}(1+z)^{-\alpha}\,dz
=
\frac{1}{p}\,
{}_2F_1(\alpha,p;p+1;-1),
\end{equation}
the preceding expression reduces exactly to the positive branch stated in the theorem.

\medskip

\noindent\textbf{B. Negative branch: $\mathbf{0<s<1}$.}

For $0<s<1$, the powers $t_1^{2(s-1)}$ and $t_2^{2(s-1)}$ are singular at the origin. Hence, before applying the binomial expansion, the singular factors must be extracted:
\begin{equation}
1+(b/a)^{-2}t_1^{2(s-1)}
=
(b/a)^{-2}t_1^{2(s-1)}
\left[
1+(b/a)^2t_1^{2(1-s)}
\right],
\end{equation}
and
\begin{equation}
1+(b/a)^2t_2^{2(s-1)}
=
(b/a)^2t_2^{2(s-1)}
\left[
1+(b/a)^{-2}t_2^{2(1-s)}
\right].
\end{equation}
After this factorization, the remaining factors are regular at the origin. Expanding the admissible factors and using
\begin{equation}
z_1=t_1^s,
\qquad
z_2=t_2^{2(1-s)},
\end{equation}
we obtain
\begin{equation}
\begin{aligned}
L^-_{(s)}
={}&
\frac{4a}{s}
\sum_{k=0}^{\infty}
\frac{(-1)^k(-1/2)_k}{k!}
\left(\frac{b}{a}\right)^{2k}
\int_0^1
z_1^{v_k^--1}
(1+z_1)^{-\alpha}
\,dz_1
\\
&+
\frac{2b}{1-s}
\sum_{k=0}^{\infty}
\frac{(-1)^k(\alpha)_k}{k!}
\int_0^1
z_2^{u_k^--1}
\left(1+\frac{a^2}{b^2}z_2\right)^{1/2}
\,dz_2 .
\end{aligned}
\end{equation}
The same Euler-type integral representations used above transform these integrals into Gauss hypergeometric functions. Therefore, the preceding expression gives precisely the negative branch stated in the theorem. Since this branch involves a binomial expansion whose resulting series is not, in general, convergent in the ordinary sense, the identity is understood in the Abel-regularized sense.

This completes the proof.
\end{proof}

\subsection{Reduction to
the supercircle ($\boldsymbol{a=b}$).}

\begin{corollary}[Supercircle case]
    For $a=b$, both branches of Theorem \ref{Thm-3.2} reduce to the supercircle
formulas
\begin{equation}\label{Eq-supercircle-positive-final}
    L_{(s)}^{+}
    =
    \frac{8a}{s}
    \sum_{j=0}^{\infty}
    \frac{(-1)^j(-1/2)_j}{j!\,v_j^{+}}\,
    {}_2F_1
    \left(
        \alpha,v_j^{+};v_j^{+}+1;-1
    \right),
    \qquad s>1
\end{equation}
and
\begin{equation}\label{Eq-supercircle-negative-final}
    L_{(s)}^{-}
    =
    \frac{8a}{s}
    \sum_{j=0}^{\infty}
    \frac{(-1)^j(-1/2)_j}{j!\,v_j^{-}}\,
    {}_2F_1
    \left(
        \alpha,v_j^{-};v_j^{-}+1;-1
    \right),
    \qquad 0<s<1.
\end{equation}
Here the notation is the same as in Theorem~\ref{Thm-3.2}. These expressions coincide with the known perimeter formulas for the Lam\'e supercircle reported in \cite{MontillaRodriguezLinares2026}.
\end{corollary}

\begin{proof}
We first consider the positive branch. Write $L_{(s)}^{+}=A_{(s)}^{+}+B_{(s)}^{+}$, where $A_{(s)}^{+}$ and $B_{(s)}^{+}$ denote, respectively, the first and second contributions in Theorem \ref{Thm-3.2}. Setting $a=b$ in the first contribution gives
\begin{equation}
A_{(s)}^{+}
=
\frac{4a}{s}
\sum_{j=0}^{\infty}
\frac{(-1)^j(-1/2)_j}{j!\,v_j^{+}}
\,{}_2F_1\!\left(\alpha,v_j^{+};v_j^{+}+1;-1\right).
\end{equation}
For the second contribution, the same specialization gives
\begin{equation}
B_{(s)}^{+}
=
\frac{2a}{s-1}
\sum_{k=0}^{\infty}
\frac{(-1)^k(\alpha)_k}{k!\,u_k^{+}}
\,{}_2F_1\!\left(-\frac12,u_k^{+};u_k^{+}+1;-1\right).
\end{equation}
Using the Gauss expansion, we obtain
\begin{equation}
B_{(s)}^{+}
=
\frac{2a}{s-1}
\sum_{k=0}^{\infty}
\frac{(-1)^k(\alpha)_k}{k!\,u_k^{+}}
\sum_{j=0}^{\infty}
\frac{(-1/2)_j(u_k^{+})_j}{(u_k^{+}+1)_j\,j!}
(-1)^j .
\end{equation}
From the definitions of $u_k^{+}$ and $v_j^{+}$, one has
\begin{equation}
\frac{1}{u_k^{+}}
\frac{(u_k^{+})_j}{(u_k^{+}+1)_j}
=
\frac{2(s-1)}{s}
\frac{1}{v_j^{+}}
\frac{(v_j^{+})_k}{(v_j^{+}+1)_k}.
\end{equation}
Substituting this identity into the preceding expression and rearranging the sums yields
\begin{equation}
B_{(s)}^{+}
=
\frac{4a}{s}
\sum_{j=0}^{\infty}
\frac{(-1)^j(-1/2)_j}{j!\,v_j^{+}}
\sum_{k=0}^{\infty}
\frac{(\alpha)_k(v_j^{+})_k}{(v_j^{+}+1)_k\,k!}
(-1)^k .
\end{equation}
The inner sum satisfies
\begin{equation}
\sum_{k=0}^{\infty}
\frac{(\alpha)_k(v_j^{+})_k}{(v_j^{+}+1)_k\,k!}
(-1)^k
=
{}_2F_1\!\left(\alpha,v_j^{+};v_j^{+}+1;-1\right).
\end{equation}
Therefore, $B_{(s)}^{+}=A_{(s)}^{+}$. Consequently,
\begin{equation}
L_{(s)}^{+}
=
A_{(s)}^{+}+B_{(s)}^{+}
=
2A_{(s)}^{+},
\end{equation}
which is precisely the positive supercircle formula stated in the corollary.

For the negative branch, write $L_{(s)}^{-}=A_{(s)}^{-}+B_{(s)}^{-}$. Setting $a=b$ in the first contribution gives
\begin{equation}
A_{(s)}^{-}
=
\frac{4a}{s}
\sum_{j=0}^{\infty}
\frac{(-1)^j(-1/2)_j}{j!\,v_j^{-}}
\,{}_2F_1\!\left(\alpha,v_j^{-};v_j^{-}+1;-1\right).
\end{equation}

The second contribution is analogous, but the boundary value of the resulting
Gauss series has to be specified. After setting $a=b$ and rearranging the
double series, the inner hypergeometric series is evaluated at $z=-1$, which
lies on the boundary of the disk of convergence. Hence, absolute convergence
and the interchange of the sums are not justified a priori for the whole range
$0<s<1$. We therefore compute the expression first at $z=-\varepsilon$, with
$0<\varepsilon<1$, where the Gauss series is absolutely convergent, and then
take the Abel limit $\varepsilon\to1^{-}$. Thus, we introduce
\begin{equation}
B_{(s)}^{-}(\varepsilon)
=
\frac{4a}{s}
\sum_{j=0}^{\infty}
\frac{(-1)^j(-1/2)_j}{j!\,v_j^{-}}
\,{}_2F_1\!\left(\alpha,v_j^{-};v_j^{-}+1;-\varepsilon\right).
\end{equation}
The same rearrangement used for the positive branch is then justified for
$0<\varepsilon<1$. Passing to the Abel limit gives
\begin{equation}
B_{(s)}^{-}
=
\lim_{\varepsilon\to1^{-}}B_{(s)}^{-}(\varepsilon)
=
A_{(s)}^{-},
\end{equation}
where the equality is understood in the Abel sense. Consequently,
\begin{equation}
L_{(s)}^{-}
=
A_{(s)}^{-}+B_{(s)}^{-}
=
2A_{(s)}^{-},
\end{equation}
which is precisely the negative supercircle formula stated in the corollary.

Thus, in the isotropic case $a=b$, both branches reduce consistently to the corresponding supercircle formulas.
\end{proof}

The preceding corollary shows that the hypergeometric formulation of
Theorem~\ref{Thm-3.2} is consistent with the isotropic limit. In fact, when $a=b$, the
two branches obtained for the general superellipse reduce exactly to the known
supercircle expressions. This agreement provides a natural check on the
structure of the formulas derived above.

\subsection{Convergence of the hypergeometric factors in the branches.}\label{subsection3.2}

Returning to Theorem~\ref{Thm-3.2}, we now examine the convergence of the
hypergeometric factors appearing in the branch representations. Since some of these
factors are evaluated at the boundary point $z=-1$, while others require
analytic continuation outside the disk of convergence of the Gauss series,
we record the relevant facts in the following proposition.
\begin{proposition}[Convergence of the hypergeometric factors]\label{Proposition3.4}
    Let $\alpha=1+1/s$ and let $u_k^\pm$, $v_k^\pm$ be defined as in (\ref{ParameterS>1}) and (\ref{Eq-parameters-minus}), respectively. Then, for every
$k\geq 0$, the following statements hold.
    \begin{itemize}
        \item[(i)] The hypergeometric factor ${}_2F_1(\alpha,v_k^+;v_k^++1;-1)$  converges conditionally for $s>1$.
        \item[(ii)]  The hypergeometric factor ${}_2F_1(\alpha,v_k^-;v_k^-+1;-1)$ diverges as an ordinary Gauss series for $0<s<1$.
        \item[(iii)]   The hypergeometric factor ${}_2F_1(-1/2,u_k^\pm;u_k^\pm+1;-a^2/b^2)$ is understood by analytic continuation, for instance through Pfaff's
    transformation. 
    \end{itemize}
\end{proposition}

\begin{proof}
We first analyze the hypergeometric factors evaluated at the boundary point
$z=-1$. Let $\text{a}=\alpha$,  $\text{b}=v_k^\pm$, and $\text{c}=v_k^\pm+1$. Then $\operatorname{Re}(\text{c}-\text{a}-\text{b})=-{1}/{s}$.

Since $z=-1$ satisfies $|z|=1$ and $z\neq 1$, the boundary convergence
criterion for the Gauss series, recalled in Appendix~\ref{AppendixA},
item~\ref{Ap4}, applies. Hence, the series converges
conditionally when
\begin{equation}
-1<\operatorname{Re}(\text{c}-\text{a}-\text{b})\leq 0,
\end{equation}
and diverges when
\begin{equation}
\operatorname{Re}(\text{c}-\text{a}-\text{b})\leq -1.
\end{equation}

For the positive branch, $s>1$, and therefore $-1<-{1}/{s}<0$. It follows that ${}_2F_1\!\left(\alpha,v_k^{+};v_k^{+}+1;-1\right)$ converges conditionally. This proves \textnormal{(i)}.

For the negative branch, $0<s<1$, and thus $-{1}/{s}<-1.$
Consequently, the defining Gauss series of
${}_2F_1\!\left(\alpha,v_k^{-};v_k^{-}+1;-1\right)$
diverges as an ordinary power series. This proves \textnormal{(ii)}. 

This statement concerns only the power-series representation at the boundary point
$z=-1$; it does not imply that the corresponding hypergeometric function is
meaningless under analytic continuation or Abel regularization.

It remains to consider the factors
${}_2F_1\!\left(-1/2,u_k^{\pm};u_k^{\pm}+1;-{a^2}/{b^2}\right)$. Set $z={a^2}/b^2$. For $z>0$, Pfaff's transformation (\ref{Pfaff}) gives
\begin{equation}
{}_2F_1\!\left(-\frac12,u_k^{\pm};u_k^{\pm}+1;-z\right)
=
(1+z)^{1/2}
{}_2F_1\!\left(
-\frac12,1;u_k^{\pm}+1;\frac{z}{1+z}
\right).
\end{equation}
Since
\begin{equation}
0<\frac{z}{1+z}<1,
\end{equation}
the transformed hypergeometric function has an absolutely convergent
power-series representation, by \eqref{2F1-def-Serie}. Moreover, from the
definitions of $u_k^{\pm}$, one has $u_k^{\pm}>0$ for $k\geq0$ in the
corresponding regimes. Hence
\[
\operatorname{Re}(u_k^{\pm}+1)>\operatorname{Re}(1)>0,
\]
so Euler's integral representation \eqref{2F1-def-Integr} also applies to
the transformed hypergeometric function.

Therefore, although the original argument $-a^2/b^2$ may lie outside the disk
of convergence of the defining Gauss series, Pfaff's transformation rewrites
the factor in terms of an equivalent hypergeometric expression whose argument
lies in the open unit interval. This proves \textnormal{(iii)}.
\end{proof}

The preceding proposition applies only to the individual Gauss hypergeometric factors appearing in (\ref{L-final-ngreater1}) and (\ref{L-final-nless1}), and not to the complete perimeter expansions themselves. Since the latter are series whose terms involve these factors, their convergence requires an independent analysis. This analysis is carried out in the next section, after rewriting the solution in a form that makes explicit the symmetry under interchange of the semi-axes.

\section{Symmetric
representation and convergence.}\label{Section4}

Under the rotation $\theta\mapsto\theta+\pi/2$, the Lamé curve is transformed
into the same curve with the semi-axes interchanged. Indeed,
\begin{equation}
r_{a,b}\left(\theta+\frac{\pi}{2}\right)
=\left(\frac{|\sin\theta|^s}{a^s}
+\frac{|\cos\theta|^s}{b^s}
\right)^{-1/s}
=r_{b,a}(\theta).
\label{Eq-rotation-exchange}
\end{equation}
Since arc length is invariant under rigid rotations, the perimeter satisfies $L_{(s)}(a,b)=L_{(s)}(b,a)$.
Thus, it is enough to assume $a\geq b>0$. The branch representations \eqref{L-final-ngreater1} and \eqref{L-final-nless1}, obtained in
Theorem~\ref{Thm-3.2}, are not manifestly symmetric under the
interchange of the semi-axes because some hypergeometric factors
appear with arguments such as $-b^2/a^2$ and $-a^2/b^2$. In this section, we first rewrite them in a symmetric
form and then use the resulting representation to study the convergence of the corresponding outer series.

As established in
Proposition~\ref{Proposition3.4}, item (iii), the latter type of
factor is not represented by a convergent Gauss series in its original form, but
it admits an equivalent Pfaff-transformed representation whose argument lies
inside the unit disk. This is precisely the transformation needed below.

Using the Gauss expansion collected in
Appendix~\ref{AppendixA}, the non-manifestly symmetric contribution can be written as
\begin{equation}
\begin{aligned}
\sum_{k=0}^\infty
\frac{(-1)^k (-1/2)_k}{k!v^\pm_k}\left(\frac{b}{a}\right)^{{2k}} &
{}_2F_1\left(\alpha,v^\pm_k;v^\pm_k+1;-1\right)\\ &=\frac{\pm s}{2(s-1)}
\sum_{k=0}^{\infty}
\frac{(-1)^k(\alpha)_k}{k!\,u_k^{\pm}}\,
{}_2F_1\left(
-\frac{1}{2},u_k^{\pm};u_k^{\pm}+1;
-\frac{b^2}{a^2}
\right),
\end{aligned}
\label{Eq-nonmanifest-term}
\end{equation}
where the upper sign corresponds to $s>1$ and the lower sign to $0<s<1$.
For $s>1$, the rearrangement leading to \eqref{Eq-nonmanifest-term} is
understood in the ordinary sense. For $0<s<1$, it is understood in the Abel
sense, replacing $(-1)^k$ by $(-\varepsilon)^k$, with $0<\varepsilon<1$, performing
the rearrangement in the convergent regime, and then taking
$\varepsilon\to1^{-}$. Substituting \eqref{Eq-nonmanifest-term} into the corresponding branch formula
of Theorem~\ref{Thm-3.2} gives
\begin{equation}
L_{(s)}^{\pm}
=
\frac{2}{\pm(s-1)}
\sum_{k=0}^{\infty}
\frac{(-1)^k(\alpha)_k}{k!\,u_k^{\pm}}
\left[
a\,{}_2F_1\left(
-\frac{1}{2},u_k^{\pm};u_k^{\pm}+1;
-\frac{b^2}{a^2}
\right)
+
b\,{}_2F_1\left(
-\frac{1}{2},u_k^{\pm};u_k^{\pm}+1;
-\frac{a^2}{b^2}
\right)
\right].
\label{eq:non-manifest-branch}
\end{equation}
Although \eqref{eq:non-manifest-branch} is symmetric under $a\leftrightarrow b$,
one of its hypergeometric arguments may lie outside the unit disk. Applying
Pfaff's transformation, as justified in
Proposition~\ref{Proposition3.4}-(iii), yields the following
representation.

\begin{theorem}[Symmetric representation]
For $s>0$, $s\neq1$, and $a\geq b>0$, the perimeter satisfies
\[
L_{(s)}(a,b)=L_{(s)}(b,a),
\]
and admits the symmetric representation
\begin{equation}
L_{(s)}^{\pm}
=\frac{2\sqrt{a^2+b^2}}{\pm(s-1)}
\sum_{k=0}^{\infty}
\frac{(-1)^k(\alpha)_k}{k!\,u_k^{\pm}}
\left[{}_2F_1\left(
-\frac{1}{2},1;u_k^{\pm}+1;
\frac{b^2}{a^2+b^2}\right)
+{}_2F_1\left(
-\frac{1}{2},1;u_k^{\pm}+1;
\frac{a^2}{a^2+b^2}\right)\right].
\label{Eq-symmetric-after-Pfaff}
\end{equation}
Moreover,
\begin{equation}
0<
\frac{b^2}{a^2+b^2}
\leq
\frac{1}{2}
\leq
\frac{a^2}{a^2+b^2}
<1.
\label{eq:transformed-arguments}
\end{equation}
Hence, for each fixed $k$, the hypergeometric factors in
\eqref{Eq-symmetric-after-Pfaff} are represented by absolutely convergent
Gauss series.
\end{theorem}

Indeed, the inequalities in \eqref{eq:transformed-arguments} follow directly
from $a\geq b>0$. Therefore, both transformed arguments lie strictly inside the
unit disk, and the absolute convergence of each hypergeometric factor follows
from the Gauss convergence criterion recalled in Appendix~\ref{AppendixA}.

This proves the symmetric representation and the termwise convergence of its
hypergeometric factors. The convergence of the outer series in $k$ is a separate
question and is studied for the positive and negative branches in the next subsections.

\subsection{Conditional convergence of the positive branch.}\label{Subsection-positive-convergence}

\begin{theorem} For $s>1$, the series representation of $L^+_{(s)}$ in Eq.~\eqref{Eq-symmetric-after-Pfaff} converges conditionally. It does not converge absolutely.   
\end{theorem}

\begin{proof}
For $s>1$, the constant prefactor in $L_{(s)}^+$ does not affect convergence, so it is sufficient to analyze the series
\begin{equation}
\sum_{k=0}^{\infty}
(-1)^k
\frac{(\alpha)_k}{k!}
\frac{1}{u_k^+}
{}_2F_1\!\left(
-\frac12,1;u_k^+ +1;
z\right).
\end{equation}
where, for clarity, we have set $z=b^2/(a^2+b^2)$ or $z=a^2/(a^2+b^2)$.

Since $u_k^+\to\infty$, the hypergeometric factor tends to $1$. Hence, the asymptotic behavior of the amplitude is determined by the product of $(\alpha)_k/k!$ and $1/u_k^+$. Using the gamma-function representation of the Pochhammer symbol
\begin{equation}
\frac{(\alpha)_k}{k!}
=
\frac{\Gamma(k+\alpha)}
{\Gamma(\alpha)\Gamma(k+1)}
\sim
\frac{k^{\alpha-1}}{\Gamma(\alpha)}
=
\frac{k^{1/s}}{\Gamma(1+1/s)}\ .
\end{equation}
Moreover,
\begin{equation}
u_k^+
=
\frac{sk+1}{2(s-1)}
\sim
\frac{s}{2(s-1)}\,k   
\end{equation}
and therefore
\begin{equation}
\frac{(\alpha)_k}{k!}
\frac{1}{u_k^+}
{}_2F_1\!\left(
-\frac12,1;u_k^+ +1;
z
\right)
\sim
\frac{2(s-1)}
{s\,\Gamma(1+1/s)}
\frac{1}{k^{(s-1)/s}}.    
\end{equation}

Since $0<(s-1)/s<1$,
the series of absolute values is asymptotic to a divergent $p$-series. Consequently, $L_{(s)}^+$ is not absolutely convergent. It remains to prove convergence in the ordinary sense. 

The preceding asymptotic estimate shows that the amplitude tends to zero. To verify eventual monotonicity, consider the ratio of two consecutive amplitudes
\begin{equation}
\frac{
\dfrac{(\alpha)_{k+1}}{(k+1)!}
\dfrac{1}{u_{k+1}^+}}{
\dfrac{(\alpha)_k}{k!}
\dfrac{1}{u_k^+}
}\frac{{}_2F_1\!\left(
-1/2,1;u_{k+1}^+ +1; z\right)}{{}_2F_1\!\left(
-1/2,1;u_k^+ +1; z\right)}
=
\left(\frac{k+\alpha}{k+1}\right)
\frac{u_k^+}{u_{k+1}^+}
\frac{{}_2F_1\!\left(
-1/2,1;u_{k+1}^+ +1; z\right)}{{}_2F_1\!\left(
-1/2,1;u_k^+ +1; z\right)}\ .
\end{equation}

Using the explicit expression for $u_k^+$, the first two factors simplify exactly as
\begin{equation}
\left(\frac{k+\alpha}{k+1}\right)
\frac{u_k^+}{u_{k+1}^+}
=
1-\frac{(1-1/s)}{k+1}\ .  
\end{equation}
For the hypergeometric quotient, the large-$u_k^+$ expansion gives
\begin{equation}
{}_2F_1\!\left(
-\frac12,1;u_k^+ +1; z \right)
\simeq
1-\frac{z}{2 u_k^+}
+
O\left(1/u_k^+\right)^2
\end{equation}
and, since $u_k^+$ grows linearly with $k$, it follows that
\begin{equation}
\frac{{}_2F_1\!\left(
-1/2,1;u_{k+1}^+ +1; z\right)}{{}_2F_1\!\left(
-1/2,1;u_k^+ +1; z\right)}
\simeq
1+O\left(1/k\right)^2\ .
\end{equation}
Consequently,
\begin{equation}
\frac{
\dfrac{(\alpha)_{k+1}}{(k+1)!}
\dfrac{1}{u_{k+1}^+}
}{
\dfrac{(\alpha)_k}{k!}
\dfrac{1}{u_k^+}
}\frac{{}_2F_1\!\left(
-1/2,1;u_{k+1}^+ +1; z\right)}{{}_2F_1\!\left(
-1/2,1;u_k^+ +1; z\right)}
\simeq
\left(
1-\frac{1-1/s}{k}
\right)
\left(
1+O\!\left(1/k\right)^2
\right)\ .
\end{equation}

Since $1-1/s>0$, the ratio is strictly smaller than $1$ for all sufficiently large $k$. Hence, the amplitudes are eventually decreasing. Therefore, the series defining $L_{(s)}^+$ is an alternating series whose amplitudes eventually decrease to zero. By the Leibniz criterion, it converges. Since absolute convergence fails, we conclude that
$L_{(s)}^+$ converges conditionally for $s>1$.
\end{proof}

\subsection{\large Abel summability of the negative branch.}\label{Subsection-negative-convergence}

\begin{theorem}
\label{thm:negative-branch-abel}
For $0<s<1$, the series representation of $L_{(s)}^{-}$ in
\eqref{Eq-symmetric-after-Pfaff} is Abel summable.
\end{theorem}

\begin{proof}
Consider the series
\begin{equation}
    \sum_{k=0}^{\infty}
    \frac{(-1)^k(\alpha)_k}{k!}\frac{1}{u_k^-}
    \,{}_2F_1\!\left(-\frac12,1;u_k^-+1;z\right)\ ,
\qquad 0<s<1.
\end{equation}
Since $u_k^-\to\infty$ as $k\to\infty$,
\begin{equation}
    {}_2F_1\!\left(-\frac12,1;u_k^-+1;z\right)\to 1
\end{equation}
and
\begin{equation}
    \frac{(\alpha)_k}{k!}\frac{1}{u_k^-}
    =\frac{2(1-s)}{s}
    \frac{\Gamma(k+\alpha)}
    {\Gamma(\alpha)\Gamma(k+2)}
    \sim
    \frac{2(1-s)}{s \Gamma(1+1/s)}\ k^{1/s-1}\ .
\end{equation}
Hence, the general term satisfies
\begin{equation}
    \left|
    \frac{(-1)^k(\alpha)_k}{k!}\frac{1}{u_k^-}
    \,{}_2F_1\!\left(-\frac12,1;u_k^-+1;z\right)
    \right|
    \sim
    \frac{2(1-s)}{s\Gamma(1+1/s)}\,k^{1/s-1}\ .
\end{equation}
Since $0<s<1$, $1/s-1>0$. Therefore, the general term does not tend to zero. By the necessary condition for convergence, the series $L_{(s)}^-$ diverges in the ordinary sense. In particular, it is neither absolutely nor conditionally convergent.

Now introduce the Abel regularization
\begin{equation}
   \sum_{k=0}^{\infty}
    \frac{(-1)^k(\alpha)_k}{k!}\frac{1}{u_k^-}
    \,{}_2F_1\!\left(-\frac12,1;u_k^-+1;z\right)\varepsilon^k,
    \qquad 0<\varepsilon<1.
\end{equation}
Since the coefficients have algebraic growth, whereas $\varepsilon^k$ decays exponentially, $L_{(s)}^-(\varepsilon)$ is absolutely convergent for every $0<\varepsilon<1$.

Using the Euler-type representation
\begin{equation}
    \frac{1}{u_k^-}\ 
    {}_2F_1\!\left(-\frac12,1;u_k^-+1;z\right)
    =
    \int_0^1
    (1-t)^{u_k^--1}(1-zt)^{1/2}\,dt.
\end{equation}
Therefore,
\begin{equation}
\begin{aligned}
   &\int_0^1
    (1-zt)^{1/2}(1-t)^{s/(2-2s)-1}
    \sum_{k=0}^{\infty}
    \frac{(\alpha)_k}{k!}
    \left[-\varepsilon\ (1-t)^{s/(2-2s)}\right]^k
    \,dt =
    \int_0^1
    \frac{(1-zt)^{1/2}(1-t)^{s/(2-2s)-1}}
    {\left[1+\varepsilon (1-t)^{s/(2-2s)}\right]^\alpha} dt 
\end{aligned}
\end{equation}
For $0<\varepsilon <1$, the integrand 
\begin{equation}
   0< \frac{(1-zt)^{1/2}(1-t)^{s/(2-2s)-1}}
    {\left[1+\varepsilon (1-t)^{s/(2-2s)}\right]^\alpha}<(1-t)^{s/(2-2s)-1}
\end{equation}
is dominated by $(1-t)^{\lambda-1}$, which is integrable on $[0,1]$ because $s/(2-2s)>0$. Hence, by the dominated convergence theorem,
\begin{equation}
    \lim_{\varepsilon\to1^-} \int_0^1
    \frac{(1-zt)^{1/2}(1-t)^{s/(2-2s)-1}}
    {\left[1+\epsilon (1-t)^{s/(2-2s)}\right]^\alpha}
    \,dt
    =
    \int_0^1
    \frac{(1-zt)^{1/2}(1-t)^{s/(2-2s)-1}}
    {\left[1+(1-t)^{s/(2-2s)}\right]^\alpha}
    \,dt .
\end{equation}
Consequently,
\begin{equation}
   \sum_{k=0}^{\infty}
    \frac{(-1)^k(\alpha)_k}{k!}\frac{1}{u_k^-}
    \,{}_2F_1\!\left(-\frac12,1;u_k^-+1;z\right)
    \overset{\mathrm{Abel}}{=}
    \int_0^1
    \frac{(1-zt)^{1/2}(1-t)^{s/(2-2s)-1}}
    {\left[1+(1-t)^{s/(2-2s)}\right]^{\alpha}}
    \,dt\ .
\end{equation}
Thus, for $0<s<1$, the series diverges in the ordinary sense but is Abel-summable.

\end{proof}
\begin{remark}
The ordinary divergence of the negative branch does not invalidate the perimeter
formula. It only shows that the series representation of $L_{(s)}^{-}$ cannot be
interpreted as an ordinary convergent series. It is instead understood in the Abel sense: after
replacing $(-1)^k$ by $(-\varepsilon)^k$, with $0<\varepsilon<1$, the regularized
series is obtained from the Euler integral representation and defines a
well-defined quantity. The Abel sum is then the limit as $\varepsilon\to1^{-}$,
which coincides with the corresponding Euler integral and hence with the
perimeter. Thus, the divergence of the unregularized outer series affects only
the mode of summation, not the validity of the formula.
\end{remark}

\section{Classical superellipses and limiting cases.} \label{Section5}

Our main objective in this section is to specialize the general perimeter formula (\ref{Eq-symmetric-after-Pfaff}) to prescribed values of \(s\), thereby deriving the corresponding arc-length expressions for the associated curves. The identification of several classical curves as particular cases of Eq.~(\ref{r_s(a,b)}) is standard in the literature on Lamé curves and superellipses and is therefore not reproduced here. Instead, we restrict our attention to less familiar limiting or degenerate cases for which a derivation from the general equation may be useful.

\begin{proposition}[Cross of Lam\'e] As $s\rightarrow 0^+$, the superellipse $\mathscr{C}_s(a,b)$ degenerates to an L-shaped
cross in each quadrant, and its perimeter satisfies $L_{(0)} = 4(a+b)$.
    
\end{proposition}

\begin{proof}
In this limit, the length is approximated by
\begin{equation}
L^-_{(0)}=
4 (a^2+b^2)^{1/2} \left[ {}_2F_1\left(-\frac12,1;1;\frac{b^2}{a^2+b^2}\right)+{}_2F_1\left(-\frac12,1;1;\frac{a^2}{a^2+b^2}\right)\right] \lim_{s\rightarrow 0^+}
\sum_{k=0}^\infty
\frac{(-1)^k}{(k+1)!}
\frac{(1/s)_k}{s}\ .
\end{equation}
Since ${}_2F_1(\text{a},\text{c},\text{c},z)=(1-z)^{-\text{a}}$, the term outside the $k$-sum reduces to $4(a+b)$. For Pochhammer's symbol, when $s$ is very small,
\begin{equation}
    (1/s)_k\simeq s^{-k}\left(1+\frac{1}{2}(k-1) k\ s\right)+O(s)^{2-k}
\end{equation}
and the series converges to
\begin{equation}
    \sum_{k=0}^\infty
\frac{(-1)^k}{(k+1)!}
\frac{(1/s)_k}{s}\simeq 1-\frac{e^{-1/s}}{2s}
\end{equation}

Evaluating the limit
\begin{equation}
  \lim_{s\rightarrow 0^+} \frac{e^{-1/s}}{s}=\lim_{\tau\rightarrow\infty}\frac{\tau}{e^\tau}\xrightarrow[\text{}]{\text{L'H\^opital}} 0\ ,
\end{equation}
we obtain
\begin{equation}\label{L_0}
    L^-_{(0)}=4 (a+b)\ .
\end{equation}
Furthermore, from the superellipse equation 
\begin{equation}
    \lim_{s\rightarrow 0^+}\frac{r\cos\theta}{a}= \lim_{s\rightarrow 0^+} \left[1-\left(\frac{r\sin\theta}{b}\right)^s\right]\quad \text{or}\quad  \lim_{s\rightarrow 0^+}\frac{r\sin\theta}{b}= \lim_{s\rightarrow 0^+} \left[1-\left(\frac{r\cos\theta}{a}\right)^s\right]
\end{equation}
we obtain 
\begin{itemize}
    \item[i)] If $\theta\in(0, \pi/2)$, Eq.~(\ref{r_s(a,b)}) reduces to
\begin{equation}
    r=\lim_{s\rightarrow 0^+} 2^{-1/s}=0\ .
\end{equation}
\item[ii)] If $\theta\in(0, \pi/2]$ and $r\neq 0$, then
\begin{equation}
    \lim_{s\rightarrow 0^+} \left[1-\left(\frac{r\sin\theta}{b}\right)^s\right]=0\Longrightarrow \lim_{s\rightarrow 0^+}\frac{r\cos\theta}{a}=0
\end{equation}
and consequently, under these constraints, the solution is $\theta=\pi/2$ with $r\in (0, b]$.
\item[iii)] If $\theta\in[0, \pi/2)$ and $r\neq 0$, then
\begin{equation}
    \lim_{s\rightarrow 0^+} \left[1-\left(\frac{r\cos\theta}{a}\right)^s\right]=0\Longrightarrow \lim_{s\rightarrow 0^+} \frac{r\sin\theta}{b}=0
\end{equation}
and we obtain the result $\theta=0$ with $r\in (0, a]$.    
\end{itemize}
Therefore, this curve corresponds to all points in the first quadrant of the plane for which
\begin{equation}
\mathscr{C}_s\xrightarrow[\text{}]{s\rightarrow 0^+}
\left\{(r, 0) /r\in (0,a]\right\}\cup\left\{(0, \theta) /\theta\in (0,\pi/2)\right\}\cup \left\{(r, \pi/2) /r\in (0,b]\right\}
\end{equation}
The curve is L-shaped in the first quadrant, with length equal to the sum of the distances from the axes to the origin, namely $(a+b)$. Including the remaining quadrants, the curve resembles a cross, and the total arc length corresponds to the value obtained in (\ref{L_0}).
\end{proof}

In other words, since Lam\'e's stars are a set of four curves concave toward the origin, the null superellipse corresponds to the collapse of Lam\'e's stars toward the origin as $s$ becomes very small. This result is less commonly emphasized in the literature; by analogy, we refer to it as Lam\'e's cross.

\begin{proposition}[The Parabolic Star] The perimeter of the parabolic star $s=1/2$ is
\begin{equation}\label{L_1/2}
L_{(1/2)}=
4\frac{(a^{3}+b^{3})}{(a^2+b^2)}+4\frac{a^{2}b^{2}}{(a^2+b^2)^{3/2}} \left( \operatorname{arcsinh}\frac{b}{a}+\operatorname{arcsinh}\frac{a}{b}\right)\ .
\end{equation}
    
\end{proposition}

\begin{proof}
This star is formed by four parabolic segments that intersect at $(a, 0)$, $(b, \pi/2)$, $(a, \pi)$, and $(b, 3\pi/2)$. Since its shape parameter is $s=1/2<1$, its length is given by
\begin{equation}\label{L1/2-1}
L^-_{(1/2)}=
8\sqrt{a^2+b^2}
\sum_{k=0}^\infty
\frac{(-1)^k(3)_k}{(k+1)!}
\left[{}_2F_1\left(-\frac12,1;\frac{k+3}{2};\frac{b^2}{a^2+b^2}\right)+{}_2F_1\left(-\frac12,1;\frac{k+3}{2};\frac{a^2}{a^2+b^2}\right)\right] .
\end{equation}

Rewriting the hypergeometric function using Euler's integral formula,
\begin{equation}
    {}_2F_1\left(-\frac12,1;\frac{k+3}{2};z\right)=\frac{(k+1)}{2}\int_0^1 dt (1-t)^{(k-1)/2}\left(1-z t\right)^{1/2}
\end{equation}
and performing the sum-to-integral substitution, equation (\ref{L1/2-1}) becomes
\begin{eqnarray}
&& L^-_{(1/2)}=
4(a^2+b^2)^{1/2}\int_0^1 \frac{dt}{(1-t)^{1/2}}\left(\sum_{k=0}^\infty
\frac{(-1)^k(3)_k}{k!} (1-t)^{k/2}\right)\nonumber\\
&&\qquad\qquad\qquad\qquad\qquad\times\left[\left(1-\frac{b^2}{a^2+b^2} t\right)^{1/2}+\left(1-\frac{a^2}{a^2+b^2} t\right)^{1/2}\right]\ .
\end{eqnarray}

Since the binomial series converges to $(1+ (1-t)^{1/2})^{-3}$, the integral yields
\begin{equation}
\begin{aligned}
\int_0^1 dt\frac{\left(1-z t\right)^{1/2}}{(1-t)^{1/2}(1+ (1-t)^{1/2})^{3}}=&-\frac{1}{4}(1-2z)+(1-z)^{3/2}\nonumber\\ &-\frac{1}{2}(1-z)z \ln\frac{1-z}{2(1+\sqrt{1-z})-z}
\end{aligned}
\end{equation}
Together with
\begin{eqnarray}
     \operatorname{arcsinh}(z)+\operatorname{arcsinh}\left(\frac{1}{z}\right)=\ln\left(1+z+\sqrt{1+z^2}+\frac{1+\sqrt{1+z^2}}{z}\right)\ ,
\end{eqnarray}
this gives exactly \eqref{L_1/2}, and the proposition follows.

\end{proof}

\begin{corollary}[The Rhombus]\label{Coroll-5.3} The perimeter at the transition value $s=1$ (rhombus) is equal to 
\begin{equation}\label{L1+=L1-}
L_{(1)}=
4\sqrt{a^2+b^2}\ .
\end{equation}
This value is recovered as the common Abel-regularized limit of the positive and
negative branches.
\end{corollary}
    
\begin{proof}

The transition value $s=1$ must be treated by a limiting procedure, since the termwise expansions of the two branches are not ordinarily convergent at this point. For this reason, we introduce an Abel factor and define the lateral value at the rhombic limit as 
\begin{equation}
    L^\pm_{(1)}=\lim_{s\to 1^\pm} \lim_{\epsilon\to 1^-} L^\pm_{(s)}(\epsilon)
\end{equation}
The two branches are written in the unified form
\begin{align}\label{Lepsilon}
 L_{(s)}^{\pm}(\epsilon)=\frac{2\sqrt{a^2+b^2}}{\pm(s-1)}\sum_{k=0}^{\infty}
\frac{(-\epsilon)^k(\alpha)_k}{k!\,u_k^{\pm}}&\Bigg[
{}_2F_1\left(-\frac{1}{2},1;u_k^{\pm}+1;\frac{b^2}{a^2+b^2}\right)
    \nonumber\\& +
{}_2F_1\left(-\frac{1}{2},1;u_k^{\pm}+1;\frac{a^2}{a^2+b^2}
        \right)\Bigg]\ .
\end{align}

Near \(s=1\), the Pochhammer factors admit the second-order expansions
\begin{equation}
\frac{(\alpha)_k}{2(1-s)u_k^-}
\simeq 1+(1-s) H_{k+1}+(1-s)^2\left[H_{k+1}+
\frac{H_{k+1}^2-H_{k+1}^{(2)}}{2}
\right]+O(1-s)^3\ ,
\end{equation}
and 
\begin{equation}
\frac{(\alpha)_k}{2(s-1)u_k^+}
\simeq 1-(s-1) H_{k}+(s-1)^2\left[H_{k}+
\frac{H_{k}^2-H_{k}^{(2)}}{2}
\right]+O(s-1)^3\ ,
\end{equation}
where the harmonic numbers and second-order harmonic numbers are defined by
\begin{equation}
H_k=\sum_{j=1}^{k}\frac1j\ ,
\qquad
H_k^{(2)}=\sum_{j=1}^{k}\frac1{j^2}\ ,
\qquad H_0=0\ .
\end{equation}
Moreover,
\begin{equation}\label{H_k+1}
    H_{k+1}=H_k+\frac{1}{k+1}\ ,\qquad H^{(2)}_{k+1}=H^{(2)}_k+\frac{1}{(k+1)^2}\ ,
\end{equation}
 The relations in Eq.~(\ref{H_k+1}) are used to handle the shifted indices that appear in the two branches.

Since $u_k^\pm\to\infty$ as $s\to1^\pm$, the hypergeometric terms can be expanded asymptotically in inverse powers of $u_k^\pm$,
\begin{equation}
\begin{aligned}
    {}_2F_1\left(-\frac12,1;u^-_k+1;\frac{b^2}{a^2+b^2}\right)&+{}_2F_1\left(-\frac12,1;u^-_k+1;\frac{a^2}{a^2+b^2}\right)\nonumber\\
&\simeq 2-\frac{(1-s)}{k+1}-\left[k-\frac{2a^2b^2}{(a^2+b^2)^2}\right]\frac{(1-s)^2}{(k+1)^2}
+O(1-s)^3
\end{aligned}
\end{equation}
and
\begin{equation}
\begin{aligned}
    &{}_2F_1\left(-\frac12, 1;u^+_k+1;\frac{b^2}{a^2+b^2}\right)+{}_2F_1\left(-\frac12,1;u^+_k+1;\frac{a^2}{a^2+b^2}\right)\nonumber\\
&\qquad\qquad\qquad\simeq 2-\frac{(s-1)}{k+1}+\left[k-\frac{2a^2b^2}{(a^2+b^2)^2}+\frac{(a^2-b^2)^2}{(a^2+b^2)^2}\right]\frac{(s-1)^2}{(k+1)^2}
+O(s-1)^3
\end{aligned}
\end{equation}

Substituting these expansions into Eq.~(\ref{Lepsilon}) and keeping terms up to second order gives the regularized expression
\begin{equation}
\begin{aligned}
L^\pm_{(s)}(\epsilon)&\simeq 4\sqrt{a^2+b^2} 
 \left[\frac{2}{1+\epsilon}\mp\frac{\ln(1+\epsilon)}{\epsilon (1+\epsilon)}(1-\epsilon)(\pm (s-1))\right.\nonumber\\ &\left. +\left(\frac{2-\ln(1+\epsilon)}{2\epsilon(1+\epsilon)}\ln(1+\epsilon)(1-\epsilon)-\frac{2a^2b^2}{(a^2+b^2)^2}\frac{\operatorname{Li}_2(-\epsilon)}{\epsilon}\right)(\pm (s-1))^2 \right]+O(\pm (s-1))^3
 \end{aligned}
\end{equation}

The sums required to evaluate this expression are
\begin{equation}
    \sum_{k=0}^\infty (-\epsilon)^k=\frac{1}{1+\epsilon}\ ,\qquad \sum_{k=0}^\infty\frac{(-\epsilon)^k}{(k+1)^2}=\frac{\operatorname{Li}_2(-\epsilon)}{\epsilon}
\end{equation}
and
\begin{equation}
   \sum_{k=0}^\infty (-\epsilon)^k H_{k}=-\frac{\ln(1+\epsilon)}{1+\epsilon}\ ,\qquad  \sum_{k=0}^\infty (-\epsilon)^k \left(H^2_{k}-H^{(2)}_k\right)=\frac{\ln^2(1+\epsilon)}{1+\epsilon}\ .
\end{equation}
They follow from the standard generating functions for harmonic numbers and from termwise integration. The dilogarithm appears naturally from the summation of terms with quadratic denominators, and its definition is recalled as
\begin{equation}
\operatorname{Li}_2(z)=\sum_{m=1}^\infty\frac{z^m}{m^2}\ ,\qquad |z|<1
\end{equation}

After performing the sums and taking the Abel limit, the first-order contribution vanishes, while the second-order term remains finite. Therefore, both lateral branches lead to the Abel-regularized expansion 
\begin{equation}
L^\pm_{(s)}\overset{\mathrm{Abel}}{=}
4\sqrt{a^2+b^2}
\left[
1+
\frac{\pi^2}{6}
\frac{a^2b^2}{(a^2+b^2)^2}
(\pm (s-1))^2
+
O(\pm (s-1))^3
\right]\ .
\end{equation}
In particular,
\begin{equation}
L_{(1)}=\lim_{s\to 1^\pm}L^\pm_{(s)}=
4\sqrt{a^2+b^2}\ ,
\end{equation}
which shows that the common limiting value is precisely the perimeter of the rhombus.

\end{proof}

Thus, the Abel-regularized representation recovers the expected rhombic perimeter at \(s=1\). The geometric interpretation of the second-order term, as well as its relation with the local minimality of the perimeter, will be discussed in Section~\ref{Section6}.

\begin{corollary}[Ellipse] For $s=2$, the formula of Theorem \ref{Thm-3.2} (positive branch) reduces to the classical elliptic perimeter
\begin{equation}\label{Elipse}
    L=4a\ \text{E}(\sqrt{1-b^2/a^2})\ ,
\end{equation}
where $\text{E}(z)$ is the complete elliptic integral of the second kind.
    
\end{corollary}

This curve belongs to the family $s>1$ and corresponds to an ordinary ellipse. As is well known, the circumference of an ellipse with semi-major and semi-minor axes $a$ and $b$ is given by the complete elliptic integral in (\ref{Elipse}).

We now show that for $s=2$ it is possible to obtain the total arc length (\ref{Elipse}) as a particular case of the superelliptic perimeter formula (\ref{Eq-symmetric-after-Pfaff}). That is,

\begin{proof}
For $s=2$, let
\begin{equation}\label{s=2}
L^+_{(2)}=
4\sqrt{a^2+b^2}
\sum_{k=0}^\infty
\frac{(-1)^k(3/2)_k}{k!(2k+1)}
\left[ {}_2F_1\left(-\frac12,1;k+\frac{3}{2};\frac{b^2}{a^2+b^2}\right)+{}_2F_1\left(-\frac12,1;k+\frac{3}{2};\frac{a^2}{a^2+b^2}\right)\right]\ .
\end{equation}
 From Euler's integral formula
\begin{equation}
    {}_2F_1\left(-\frac12,1;k+\frac{3}{2};z\right)=\frac{(2k+1)}{2}\int_0^1 dt\ (1-t)^{k-1/2}\left(1-z t\right)^{1/2}
\end{equation}
and the binomial series
\begin{equation}
    \sum_{k=0}^\infty
\frac{(-1)^k(3/2)_k}{k!} (1-t)^{k}=(2-t)^{-3/2}
\end{equation}
we obtain
\begin{eqnarray}\label{L2}
L^+_{(2)}=
2(a^2+b^2)^{1/2}\int_0^1 \frac{dt}{(1-t)^{1/2}(2-t)^{3/2}}\left[\left(1-\frac{b^2}{a^2+b^2} t\right)^{1/2}+\left(1-\frac{a^2}{a^2+b^2} t\right)^{1/2}\right]
\end{eqnarray}

Now, making the substitution $t=1-\tan^2\phi$  in the first integral,
\begin{equation}
    \int_0^1 \frac{dt}{(1-t)^{1/2}(2-t)^{3/2}}\left(1-\frac{b^2}{a^2+b^2} t\right)^{1/2}=\frac{2a}{(a^2+b^2)^{1/2}}\int_0^{\pi/4} d\phi\ \left[1-\left(1-\frac{b^2}{a^2}\right)\sin^2\phi\right]
\end{equation}
and $t=1-\cot^2\phi$ in the second integral,
\begin{equation}
    \int_0^1 \frac{dt}{(1-t)^{1/2}(2-t)^{3/2}}\left(1-\frac{a^2}{a^2+b^2} t\right)^{1/2}=\frac{2a}{(a^2+b^2)^{1/2}}\int_{\pi/4}^{\pi/2} d\phi\ \left[1-\left(1-\frac{b^2}{a^2}\right)\sin^2\phi\right]
\end{equation}
It follows that the perimeter (\ref{s=2}) reduces to the complete elliptic integral of the second kind,
\begin{eqnarray}
  \text{E}(z)=  \int_0^{\pi/2} d\phi\ \left(1-z^2 \sin^2\phi\right)\ ,
\end{eqnarray}
Therefore, the expression for the perimeter reduces exactly to formula~(\ref{Elipse}). This proves the corollary.

\end{proof}

\begin{proposition}[Rectangle] As $s\rightarrow\infty$, the superellipse $\mathscr{C}_s(a,b)$ approaches the rectangle
$[-a, a]\times [-b, b]$, and
\begin{equation}
    L_{(\infty)}=4(a+b) .
\end{equation}
\end{proposition}

\begin{proof} To determine the corresponding perimeter, we start from the following expression
\begin{equation}
L^+_{(s)}=4 \sqrt{a^2+b^2}
\sum_{k=0}^\infty
\frac{(-1)^k(\alpha)_k}{k!(s k+1)}\left[ {}_2F_1\left(-\frac12,1;u^+_k+1;\frac{b^2}{a^2+b^2}\right)+{}_2F_1\left(-\frac12,1;u^+_k+1;\frac{a^2}{a^2+b^2}\right)\right]\ .
\end{equation}
Observe that the first term of the series does not contain factors of the form $k s$, whereas such products appear in all subsequent terms. Since, for $k>1$, these factors remain well behaved as $s \rightarrow \infty$, it is natural to separate the first term from the remainder of the summation
\begin{eqnarray}
L^+_{(\infty)}=4\sqrt{a^2+b^2}\left[{}_2F_1\left(-\frac12,1;1;\frac{b^2}{a^2+b^2}\right)+{}_2F_1\left(-\frac12,1;1;\frac{a^2}{a^2+b^2}\right)\right]&&\nonumber\\
+\lim_{s\rightarrow\infty}\frac{\ \sqrt{a^2+b^2}}{s }\sum_{k=1}^\infty
\frac{(-1)^k}{k}\left[ {}_2F_1\left(-\frac12,1;\frac{k}{2}+1;\frac{b^2}{a^2+b^2}\right)\right.&&\nonumber\\ \left.+{}_2F_1\left(-\frac12,1;\frac{k}{2}+1;\frac{a^2}{a^2+b^2}\right)\right]&&\ .
\end{eqnarray}

The first term, together with its symmetric counterpart, reduces to 
\begin{equation}
    {}_2F_1\left(-\frac12,1;1;\frac{b^2}{a^2+b^2}\right)+{}_2F_1\left(-\frac12,1;1;\frac{a^2}{a^2+b^2}\right)=\frac{a+b}{\quad \sqrt{a^2+b^2}}
\end{equation}
where (\ref{F(accz)}) has been used, whereas the limit of the second term cannot be determined without first establishing the convergence of the series
\begin{equation}\label{Sum[2F1,k>1]}
\sum_{k=1}^{\infty}\frac{(-1)^k}{k}
\,{}_2F_1\!\left(-\frac12,1;\frac{k}{2}+1;z\right),
\qquad 0<z<1\ .
\end{equation}
This is done using the Leibniz criterion and a comparison argument.

We begin by considering 
\begin{equation}
\frac{1}{k}
{}_2F_1\!\left(-\frac12,1;\frac{k}{2}+1;z\right)
=
\frac12\int_0^1
(1-t)^{k/2-1}(1-zt)^{1/2}\,dt.
\end{equation}
Since the integrand in Euler's integral representation is positive on \((0,1)\), we have
\begin{equation}
\frac{1}{k}
{}_2F_1\!\left(-\frac12,1;\frac{k}{2}+1;z\right)>0\ .
\end{equation}
Moreover, for every \(t\in(0,1)\),
\begin{equation}
(1-t)^{(k+1)/2-1}
<
(1-t)^{k/2-1}\ .
\end{equation}
Therefore,
\begin{equation}
\frac{1}{k+1}
{}_2F_1\!\left(-\frac12,1;\frac{k+1}{2}+1;z\right)
<
\frac{1}{k}
{}_2F_1\!\left(-\frac12,1;\frac{k}{2}+1;z\right),
\end{equation}
so the sequence of positive coefficients is decreasing.

On the other hand, since \(0<(1-zt)^{1/2}\le 1\), it follows that
\begin{equation}
0<
\frac{1}{k}
{}_2F_1\!\left(-\frac12,1;\frac{k}{2}+1;z\right)
\le
\frac12\int_0^1(1-t)^{k/2-1}\,dt
=
\frac1k.
\end{equation}
Thus,
\begin{equation}
\frac{1}{k}
{}_2F_1\!\left(-\frac12,1;\frac{k}{2}+1;z\right)
\to0\ , \qquad k\to\infty\ .
\end{equation}
The Leibniz criterion then implies that the series converges. 

To determine the nature of the convergence, observe that
\begin{equation}
{}_2F_1\!\left(-\frac12,1;\frac{k}{2}+1;z\right)\to1\ ,
\qquad k\to\infty\ ,
\end{equation}
because the kernel $(k/2)(1-t)^{k/2-1}$ has unit mass and concentrates at \(t=0\). Consequently,
\begin{equation}
\frac{1}{k}
{}_2F_1\!\left(-\frac12,1;\frac{k}{2}+1;z\right)
\sim \frac1k.
\end{equation}
Therefore, the series (\ref{Sum[2F1,k>1]}) diverges by comparison with the harmonic series. Hence, the original series converges conditionally. Consequently, the series defines a finite constant, and therefore
\begin{equation}
\lim_{s\to\infty}
\frac{\sqrt{a^2+b^2}}{s}
\sum_{k=1}^{\infty}\frac{(-1)^k}{k}\,
{}_2F_1\!\left(-\frac12,1;\frac{k}{2}+1;z\right)
=0\ .
\end{equation}
Finally, we obtain
\begin{equation}
    L_{(\infty)}=4(a+b).
\end{equation}
\end{proof}
 
Since $L_{(s)}$ is continuous and 
the endpoint values coincide, the generalized Weierstrass theorem 
guarantees the existence of global extrema for $s>0$. If the 
perimeter function is nonconstant, a Rolle-type argument further implies 
the existence of at least one interior stationary point. In the next 
section, we prove that the relevant extremum is attained at the rhombic 
case $s=1$.

\section{The minimal perimeter in the Lam\'e family.}\label{Section6}

\begin{theorem}[Global minimality of the rhombus]\label{Thm-6.1}
Let $a,b>0$ be fixed, and let 
$\{\mathscr{C}_{s}(a,b):s>0\}$ be the Lam\'e family. Then, the perimeter
$L_{(s)}$ satisfies
\begin{equation}
L_{(s)}\geq L_{(1)}=4\sqrt{a^2+b^2},
\end{equation}
for every $s>0$. The minimum is attained only at $s=1$; equivalently, the rhombus
$\mathscr{C}_{1}(a,b)$ is the unique global minimizer of the perimeter within
the Lam\'e family. Moreover, $s=1$ is a strict local minimum of the
Abel-regularized analytic representation of the two hypergeometric branches.
\end{theorem}

\begin{proof}
By the axial symmetries of the Lam\'e curve, it is enough to consider the arc
contained in the first quadrant. For every $s>0$, this arc joins the fixed
endpoints $(a,0)$ and $(0,b)$. Therefore, if $\gamma_s$ denotes this arc, its
length satisfies the elementary geometric inequality
\begin{equation}
\ell(\gamma_s)
\geq
\|(a,0)-(0,b)\|
=
\sqrt{a^2+b^2}.
\end{equation}
Since the full curve consists of four congruent arcs, it follows that
\begin{equation}
L_{(s)}=4\ell(\gamma_s)\geq 4\sqrt{a^2+b^2}.
\end{equation}

Equality in the length-distance inequality occurs only when
$\gamma_s$ is the straight segment joining $(a,0)$ and $(0,b)$. In the Lam\'e
family, this happens precisely for $s=1$; see Corollary~\ref{Coroll-5.3}.
Thus, the complete curve is the rhombus with vertices $(\pm a,0)$ and
$(0,\pm b)$. This proves both the global minimality and the uniqueness of the minimizer.

It remains to describe the analytic behavior of the branch representations near
the minimizing value. From Corollary~\ref{Coroll-5.3}, the Abel-regularized
expansions of the two branches satisfy
\begin{equation}
L_{(s)}^{\pm}
\overset{\mathrm{Abel}}{\simeq}
4\sqrt{a^2+b^2}
\left[
1+
\frac{\pi^2}{6}
\frac{a^2b^2}{(a^2+b^2)^2}
\bigl(s-1\bigr)^2
+
O\bigl(s-1\bigr)^3
\right].
\end{equation}
Hence, both branches have the common limiting value
\begin{equation}
\lim_{s\to1^\pm}L_{(s)}^\pm
=
4\sqrt{a^2+b^2}
=
L_{(1)}.
\end{equation}
Moreover, the absence of a linear term implies that the first lateral
derivatives vanish at $s=1$, in the Abel-regularized sense
\begin{equation}
\lim_{s\to1^\pm}\frac{dL_{(s)}^\pm}{ds}=0.
\end{equation}
Finally, the first nonzero correction is quadratic with positive coefficient,
namely
\begin{equation}
\frac{\pi^2}{6}
\frac{a^2b^2}{(a^2+b^2)^2}>0.
\end{equation}
Therefore, the Abel-regularized branch representation has a strict local
minimum at $s=1$. Together with the preceding geometric argument, this confirms
that the rhombus is the unique global minimizer of the perimeter in the Lam\'e
family.
\end{proof}

\begin{remark}

The second-order terms represent the first nontrivial correction induced by a perturbation of the Lam\'e exponent $s$ around the rhombic case $s=1$. In particular, the geometric factor $a^2b^2/(a^2+b^2)^2$ measures the balance between the two semi-axes. This factor attains its maximum when $a=b$ and tends to zero when one semi-axis dominates the other, namely as $a/b\to0$ or $a/b\to\infty$. Since the full quadratic coefficient is positive, the expansion shows that, at least in this local Abel-regularized sense, the perimeter increases quadratically as the curve departs from the case $s=1$. Consequently, this coefficient quantifies the deviation of the Lam\'e curve perimeter from the rhombus perimeter under small perturbations of the exponent $s$ around $1$.
\end{remark}

This analytic conclusion is supported by the behavior displayed in Fig.~\ref{Tercera imagen} for $L_{(s)}$ versus $s$. As $s$ approaches $1$ from the left, the perimeter decreases toward the rhombic value, whereas for $s>1$ it increases away from that same value. Hence, $s=1$ represents the minimum of the curve $L_{(s)}$. In the plotted case, this minimum is attained at the rhombus, while the neighboring cases, such as the parabolic case, the astroid, the ellipse, and the quasi-rectangular limit, all have larger perimeter.

\begin{figure}[h]
\centering
\includegraphics[width=1.0\linewidth]{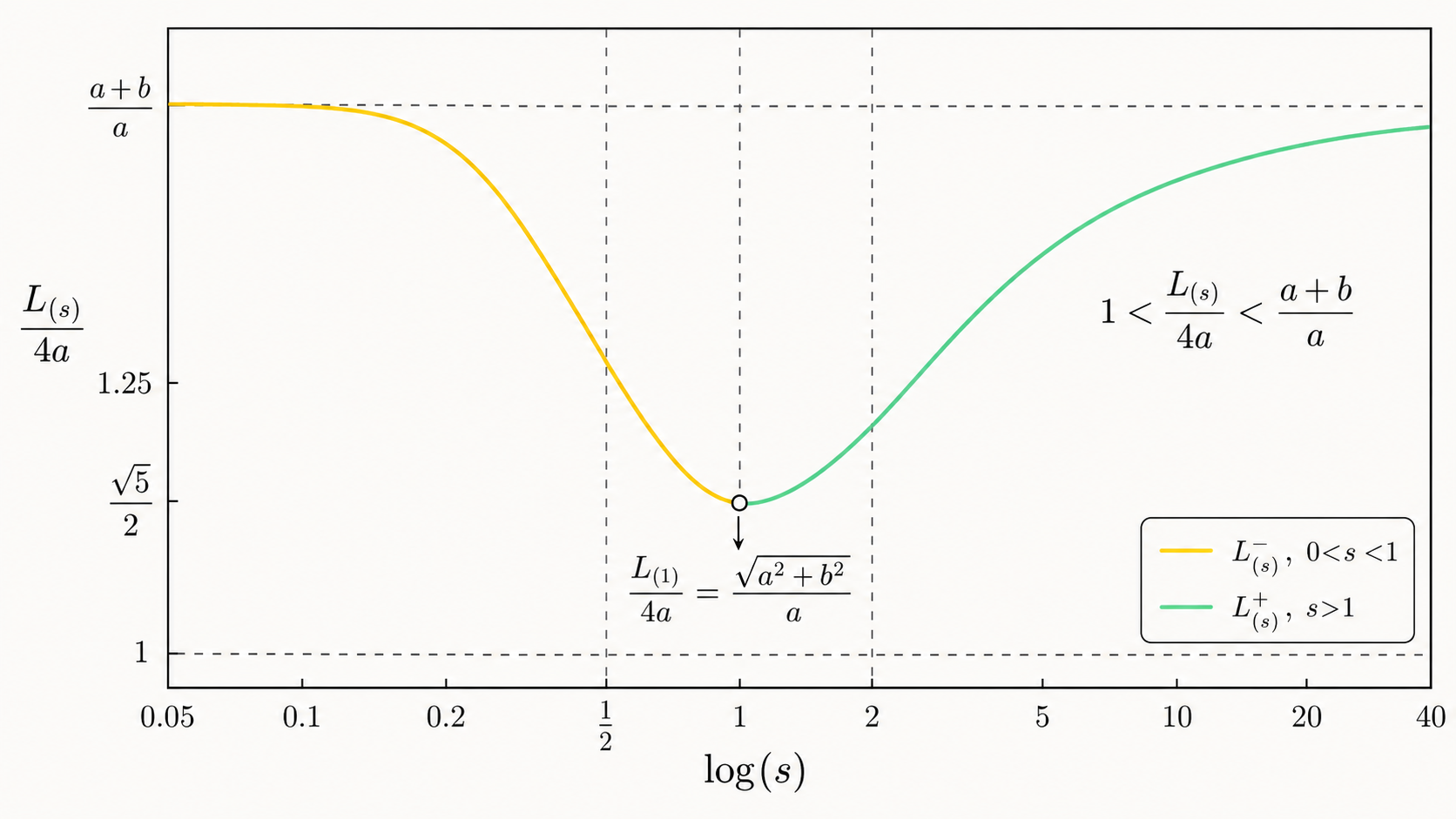}
\caption{Perimeter of superellipses $L_{(s)}$ depicted for $a=2b$ and $b=1$. A semi-logarithmic scale is used on the horizontal axis to better visualize the behavior for small and large values of $s$. The minimum at $s=1$ (Theorem~\ref{Thm-6.1}) is clearly visible.}
\label{Tercera imagen}
\end{figure}

\section{Conclusions and remarks.}\label{Conclusions}

In this paper, we have analytically established, in accordance with Lemmas~\ref{Lemma-2.5} and~\ref{Lemma-2.6}, that the parameter $s$ determines both the local regularity at the vertices and the concavity regime of the superellipse. For $s>1$, the vertices $\theta_k=k\pi/2$ are regular $C^{1}$-points, and the four quadrantal arcs are convex. At the transition value $s=1$, the curve reduces to the rhombic case: the curvature vanishes along the straight sides, whereas the vertices become angular points with finite but discontinuous one-sided derivatives. For $0<s<1$, the regularity at the vertices breaks down, producing cusps with infinite one-sided derivatives of opposite signs, and each quadrantal arc is concave toward the origin. Thus, $s=1$ separates the concave cuspidal regime from the smooth convex regime.

The geometric analysis also identifies the transition direction 
$\theta_0=\arctan(b/a)$, which provides a natural partition of the first 
quadrant into two sectors and serves as a reference direction for comparing 
the behavior above and below the diagonal of the associated rectangle, as 
described in Remark~2.7. Starting from this geometric structure, the 
hypergeometric branch representation established in Theorem~3.2 gives an 
exact analytic expression for the perimeter for $s>0$, $s\neq 1$. The 
subsequent analysis of convergence and symmetry leads to the symmetric 
representation of Theorem~4.1, which makes explicit the invariance of the 
perimeter under the exchange of the semi-axes $a$ and $b$. Within this 
framework, several distinguished cases are recovered: the limiting cross of 
Lam\'e as $s\to 0^{+}$, with perimeter $4(a+b)$, in Proposition~5.1; the 
parabolic star for $s=1/2$ in Proposition~5.2; the rhombic transition value 
$s=1$, obtained as the common Abel-regularized limit of both branches in 
Corollary~5.3; the classical elliptic perimeter for $s=2$ in 
Corollary~5.4; and the rectangular limit as $s\to\infty$, again with 
perimeter $4(a+b)$, in Proposition~5.5. Thus, the analytic perimeter 
formulas, their symmetry properties, and the limiting cases are consistent 
with the geometric transition described by the family: from concave cuspidal 
arcs for $0<s<1$, through the rhombic case $s=1$, to convex smooth arcs for 
$s>1$.

\begin{remark}
Although $L_{(\infty)}=L_{(0)}$, the limiting curves are geometrically
distinct: $\mathscr{C}_{(s)}$ degenerates to a cross as $s\to 0^+$ and to a rectangle as $s\to\infty$. This
suggests a natural compactification of the parameter space in which the endpoints
are identified in length but not in shape. We hope to address this issue in a future publication.
\end{remark}

\begin{remark}[Linear superellipse and lower bound] Setting $b=0$, every superellipse degenerates to a horizontal line
with 
\begin{equation}\label{L_b=0}
L_{(s)}=4a\ ,\quad \forall s\in (0, \infty)\ .
\end{equation}
That is, if $b=0$, the length of any superellipse is represented as
\begin{equation}
L^+_{(s)}=
4a
\sum_{k=0}^\infty
\frac{(-1)^k(\alpha)_k}{k!\ (s\ k+1)}
+4a
\sum_{k=0}^\infty
\frac{(-1)^k(\alpha)_k}{k!\ (s\ k+1)}
\ {}_2F_1\left(-1/2,1;u^+_k+1;1\right)
\end{equation}
and
\begin{equation}
L^-_{(s)}=
4a
\sum_{k=0}^\infty
\frac{(-1)^k(\alpha)_k}{(k+1)!}
\frac{1}{s}+4a
\sum_{k=0}^\infty
\frac{(-1)^k(\alpha)_k}{(k+1)!}
\frac{1}{s}\ {}_2F_1\left(-1/2,1;u^-_k+1;1\right)\ .
\end{equation}
Using the identities collected in Appendix~\ref{Appendix_C}, one obtains the sharp two-sided estimate
\begin{equation}
  4a<L_{(s)}<4(a+b)\ .
\end{equation}

\end{remark}

On the other hand, the polar equation for the superellipse in the first quadrant can be rearranged as $r(b^s\cos^s\theta+a^s\sin^s\theta)^{1/s}=ab$.
 When $b=0$, the equation has two solutions. The first is the point $r=0$ for all $\theta\neq 0$. The second corresponds to $\theta=0$ for all $r\neq 0$. That is, the resulting curve is given by
\begin{equation}
\mathscr{C}_s\xrightarrow[\text{}]{b\rightarrow 0}
\left\{(r, 0) /r\in (0,a]\right\}\cup\left\{(0, \theta) /\theta\in (0,\pi/2]\right\}
\end{equation}
In the second and third quadrants, we obtain two horizontal lines $\theta=\pi$ (one for each quadrant), and in the fourth quadrant, we obtain another horizontal line, $\theta=2\pi$. Thus, the superellipse degenerates into a linear superellipse at $b=0$, with length given by (\ref{L_b=0}): namely, a horizontal line on the polar axis, independent of $s$. Therefore, in addition to interpolating between the cross and the rectangle, the length of the superellipse has a lower bound given by the linear superellipse.

\appendix
\section*{\centering{Appendix}}

\section{Gauss hypergeometric function}\label{AppendixA}
\numberwithin{equation}{section}
\setcounter{equation}{0}

We recall the basic facts about the Gauss hypergeometric function used throughout the paper \cite{abramowitz1964}. In its power-series representation, it is defined by
\begin{equation}\label{2F1-def-Serie}
{}_2F_1(\text{a},\text{b};\text{c};z)=
\sum_{k=0}^{\infty}
\frac{(\text{a})_k(\text{b})_k}{(\text{c})_k}
\frac{z^k}{k!},
\qquad
c\notin\mathbb Z{\le 0},
\end{equation}
where $(\text{q})_k=\Gamma(\text{q}+k)/\Gamma(\text{q})$ denotes the Pochhammer symbol. If either a or b is a non-positive integer, the series terminates and defines a polynomial. Otherwise, its radius of convergence is one. The convergence of the power series is characterized as follows:
\begin{enumerate}
\item If $|z|<1$, the series converges absolutely.
\item If $|z|>1$, the series diverges, except in the terminating cases.

\item If $z=1$, the series converges precisely when $\operatorname{Re}(\text{c}-\text{a}-\text{b})>0$, apart from terminating cases. Under this condition, Gauss's summation formula gives
\begin{equation}
{}_2F_1(\text{a},\text{b};\text{c};1)
=
\frac{\Gamma(\text{c})\Gamma(\text{c}-\text{a}-\text{b})}
{\Gamma(\text{c}-\text{a})\Gamma(\text{c}-\text{b})}.
\end{equation}

\item\label{Ap4} If $|z|=1$, and $z\neq 1$,
the series converges absolutely when $\operatorname{Re}(\text{c}-\text{a}-\text{b})>0$, and converges conditionally when $-1<\operatorname{Re}(\text{c}-\text{a}-\text{b})\le 0$.
\end{enumerate}

For real $z$, the condition $|z|=1$ with $z\neq 1$ reduces to $z=-1$. Therefore, this case must not be confused with $z>1$, which lies outside the disk of convergence.

Outside the preceding cases, the power series does not provide an ordinary convergent representation of the hypergeometric function. Nevertheless, the function may still be defined by analytic continuation.

A standard continuation is provided by Euler's integral representation. Under the restrictions
$\operatorname{Re}(\text{c})>\operatorname{Re}(\text{b})>0$, and $|\arg(1-z)|<\pi$, one has
\begin{equation}\label{2F1-def-Integr}
{}_2F_1(\text{a},\text{b};\text{c};z)=\frac{\Gamma(\text{c})}
{\Gamma(\text{b})\Gamma(\text{c}-\text{b})}
\int_0^1 t^{\text{b}-1} (1-t)^{\text{c}-\text{b}-1}
(1-zt)^{-\text{a}}\,dt .
\end{equation}

The condition $|\arg(1-z)|<\pi$ fixes the branch of the complex power and excludes the standard cut $z\in[1,\infty)$ when $z$ is real. In particular, if $0<z<1$, then $1-z>0$, and the condition is automatically satisfied. By contrast, for real $z>1$, one has $1-z<0$, so $1-z$ lies on the negative real axis and $|\arg(1-z)|=\pi$. Hence, Euler's integral representation is not valid there as an ordinary real integral. Equivalently, the factor $(1-zt)^{-\text{a}}$ has a singular point at $t=1/z\in(0,1)$. In such cases, the hypergeometric function must be understood by analytic continuation, for instance through boundary values from the upper or lower half-plane, or by means of transformation formulas. This obstruction does not occur for an argument of the form $-z$, with $z>1$, since $1-(-z)=1+z>0$, and therefore the branch condition is again satisfied.

The elementary value at the origin is
\begin{equation}
{}_2F_1(\text{a},\text{b};\text{c};0)=1.
\end{equation}
This follows directly from the power series, since all terms with $k\ge 1$ vanish at $z=0$. In general, this identity does not extend to $z\neq 0$. The hypergeometric function is identically equal to one only in degenerate cases such as
a$=0$ or b$=0$, for which the series reduces to its constant term.

The Pfaff transformation is
\begin{equation}\label{Pfaff}
{}_2F_1(\text{a},\text{b};\text{c};z)=(1-z)^{-\text{a}} {}_2F_1\left(\text{a},\text{c}-\text{b};\text{c};\frac{z}{z-1}\right),
\qquad
\text{c}\notin\mathbb Z{\le 0}.
\end{equation}
This identity is useful when the original argument is outside the unit disk but the transformed argument lies inside it. For example, for an argument of the form $-z$, with $z>1$, one obtains
\begin{equation}
\left|
\frac{-z}{-z-1}
\right|=\frac{z}{z+1}
<1.
\end{equation}

Finally, the binomial reduction follows as the particular case b$=$c of Pfaff's transformation. Thus,
\begin{equation}\label{F(accz)}
{}_2F_1(\text{a},\text{c};\text{c};z)=(1-z)^{-\text{a}}.
\end{equation}

This collection of formulas provides the basic hypergeometric framework required in the paper: the power series determines the convergence regimes, Euler's integral representation gives an analytic continuation beyond the original disk of convergence, and Pfaff's transformation allows one to rewrite the argument in a region where a convergent series representation is available.

\section{Auxiliary identities for the limiting sums}\label{Appendix_C}
Let \(s>0\), \(s\neq 1\), and let \(\alpha=1+1/s\). The auxiliary identities used in the limiting analysis follow from Abel summation and from Gauss's evaluation formula for the hypergeometric function at unity.
\begin{itemize}
    \item For \(0<\epsilon<1\), using
\[
\frac{1}{sk+1}=\int_0^1 x^{sk}\,dx,
\]
and applying the binomial expansion, we obtain
\[
\lim_{\epsilon\to1^{-}}\sum_{k=0}^{\infty}
\frac{(-1)^k(\alpha)_k}{k!(sk+1)}\epsilon^k
=
 \lim_{\epsilon\to1^{-}} \int_0^1 (1+\epsilon x^s)^{-\alpha}\,dx .
\]
Taking \(\epsilon\to1^{-}\) in the Abel sense, it follows that
\begin{equation}\label{Series1}
\sum_{k=0}^{\infty}
\frac{(-1)^k(\alpha)_k}{k!(sk+1)}
=
2^{-1/s}.
\end{equation}
\item Similarly, using
\[
\frac{1}{k+1}=\int_0^1 x^k\,dx,
\]
we have
\[
\lim_{\epsilon\to1^{-}} \frac1s
\sum_{k=0}^{\infty}
\frac{(-1)^k(\alpha)_k}{(k+1)!}r^k
=
\lim_{\epsilon\to1^{-}} \frac1s
\int_0^1 (1+\epsilon x)^{-\alpha}\,dx .
\]
Letting \(\epsilon\to1^{-}\), we conclude that
\begin{equation}\label{Series2}
\frac1s
\sum_{k=0}^{\infty}
\frac{(-1)^k(\alpha)_k}{(k+1)!}
=
1-2^{-1/s}.
\end{equation}
\end{itemize}
It remains to evaluate the hypergeometric factors. Gauss's formula gives
\[
{}_2F_1\left(-\frac12,1;u^\pm_k+1;1\right)
=
\frac{\Gamma(u^\pm_k+1)\Gamma\left(u^\pm_k+\frac12\right)}
{\Gamma\left(u^\pm_k+\frac32\right)\Gamma(u^\pm_k)}
=
\frac{u^\pm_k}{u^\pm_k+\frac12},
\]
where the condition \(u^\pm_k+1/2>0\) is satisfied for \(s>0\). Since
\[
u^+_k+\frac12
=
\frac{s(k+1)}{2(s-1)},\qquad u^-_k+\frac12
=
\frac{sk+1}{2(1-s)},
\]
\begin{itemize}
    \item For the positive branch, 
\[
{}_2F_1\left(-\frac12,1;u^+_k+1;1\right)
=
\frac{sk+1}{s(k+1)},
\]
we get
\begin{equation}\label{Serie3}
\sum_{k=0}^{\infty}
\frac{(-1)^k(\alpha)_k}{k!(sk+1)}
{}_2F_1\left(-\frac12,1;u^+_k+1;1\right)
=
1-2^{-1/s}.
\end{equation}
\item For the negative branch, Gauss's formula gives
\[
{}_2F_1\left(-\frac12,1;u^-_k+1;1\right)
=
\frac{s(k+1)}{sk+1}.
\]
Consequently,
\begin{equation}\label{Series4}
\frac1s
\sum_{k=0}^{\infty}
\frac{(-1)^k(\alpha)_k}{(k+1)!}
{}_2F_1\left(-\frac12,1;u^-_k+1;1\right)
=
2^{-1/s}.
\end{equation}
\end{itemize}


\end{document}